\documentclass[11pt]{article}
\usepackage{amsmath,amssymb,stmaryrd,epsfig}
\usepackage{float,tikz,subfig}
\newcommand{\comment}[1]{}

\newcommand{\reals}{\mbox{$\mathbb R$}}

\newcommand{\nats}{\mbox{$\mathbb N$}}
\newcommand{\ints}{\mbox{$\mathbb Z$}}

\newcommand{\spn}{{\mathop{\mathrm{span}}\nolimits}}

\newcommand{\conv}{{\mathop{\mathrm{conv}}\nolimits}}

\newcommand{\proj}{{\mathop{\mathrm{Proj}}\nolimits}}
\newcommand{\supp}{{\mathop{\mathrm{supp}}\nolimits}} 

\DeclareRobustCommand{\stirling}{\genfrac\{\}{0pt}{}} 

\usepackage{amssymb}

\def\squarebox#1{\hbox to #1{\hfill\vbox to #1{\vfill}}}
\def\qed{\hspace*{\fill}
        \vbox{\hrule\hbox{\vrule\squarebox{.667em}\vrule}\hrule}\smallskip}
\newenvironment{proof}{\begin{trivlist}
  \item[\hspace{\labelsep}{\em\noindent Proof.~}]
  }{\qed\end{trivlist}}

\newtheorem{lemma}{Lemma}[section]
\newtheorem{theorem}[lemma]{Theorem}
\newtheorem{corollary}[lemma]{Corollary}
\newtheorem{proposition}[lemma]{Proposition}
\newtheorem{claim}[lemma]{Claim}
\newtheorem{observation}[lemma]{Observation}
\newtheorem{definition}[lemma]{Definition}

\def\squareforqed{\hbox{\rlap{$\sqcap$}$\sqcup$}}
\def\qed{\ifmmode\squareforqed\else{\unskip\nobreak\hfil
\penalty50\hskip1em\null\nobreak\hfil\squareforqed
\parfillskip=0pt\finalhyphendemerits=0\endgraf}\fi}

\newlength{\tablength}
\newlength{\spacelength}
\newcommand{\tabstar}{\hspace*{\tablength}}
\newcommand{\spacestar}{\hspace*{\spacelength}}
\def\obeytabs{\catcode`\^^I=\active}
{\obeytabs\global\let^^I=\tabstar}
{\obeyspaces\global\let =\spacestar}
\newenvironment{display}{\begingroup\obeylines\obeyspaces\obeytabs}{\endgroup}
\newenvironment{prog}{\begin{display}\parskip0pt\sf}{\end{display}}

\title{On a special class of general permutahedra} 

\author{
{\sl Geir Agnarsson} 
\thanks{Department of Mathematical Sciences,
George Mason University,
MS 3F2, 4400 University Drive,
Fairfax, VA -- 22030, USA,
{\tt geir@math.gmu.edu}}
}

\date{}

\begin{document}

\maketitle

\begin{abstract}
Minkowski sums of simplices in 
${\mathbb{R}}^n$ form an interesting class of polytopes
that seem to emerge in various situations. In this paper
we discuss the Minkowski sum of the 
simplices $\Delta_{k-1}$ in ${\mathbb{R}}^n$ where $k$ and $n$ are
fixed, their flags and some of their face lattice structure.
In particular, we derive a closed formula for 
their {\em exponential generating flag function}.
These polytopes are simple, include both the simplex $\Delta_{n-1}$
and the permutahedron $\Pi_{n-1}$, and form a Minkowski basis
for more general permutahedra.

\vspace{3 mm}


\vspace{2 mm}

\end{abstract}

\section{Introduction and motivation}
\label{sec:intro}

The Minkowski sum of simplices yields an important
class of polytopes that includes and generalizes many known
polytopes. For related references and some of the history of the
significance of Minkowski sums of simplices we refer to
the introduction in~\cite{Minkowski}. In~\cite{MinkPoly} a 
closed formula for the $\ell$-flag polynomial 
(See Definition~\ref{def:flag-poly})
for an arbitrary Minkowski sum of $k$ simplices is derived.
In particular, this yields a closed formula for the $f$-vectors of
generalized associahedra from~\cite{Postnikov}. 
This mentioned formula, however, 
is in terms of the {\em master polytope} $P(k)$, a $(2^k-2)$-dimensional 
polytope, the structure of which little is known about except when 
$k\leq 2$~\cite{Minkowski},~\cite{MinkPoly}.
In this paper we focus on the family of Minkowski sum of the 
simplices of a fixed dimension. These polytopes are interesting
for a variety of reasons. We mention a few here without attempting
to be exhaustive: 
(i) The polytopes in this family are all simple, and so 
they have a nice enumeration of their flags of arbitrary length
as we will see shortly (see Lemma~\ref{lmm:simple}). 
(ii) This family forms a chain, or an incremental bridge, between the simplex
of a given dimension and the standard permutahedron of the same
dimension, where each step, or link, is between two such simple
polytopes that differ minimally, as we will see in Section~\ref{sec:our-fam}.
(iii) Each polytope in this family is symmetric with respect
to permutation of their coordinates, like the simplex and the standard
permutahedron. In fact, they make up a subclass of the class of
generalized permutohedra studied in~\cite{Postnikov} and~\cite{Post-Vic-Laur},
something we will discuss in more detail in Section~\ref{sec:gen-perm}.
(iv) By contracting each face formed by vertices of identical positive 
support of any polytope of this family, 
one obtains a hypersimplex; a particular matroid base polytope 
(or matroid basis polytope) of the uniform matroid
formed by all subsets of a fixed cardinality (the rank of the matroid) from
a given ground set. In fact, each matroid base polytope of a matroid
of a given rank is contained in a ``mother''-hypersimplex, that is, 
its vertices are among the vertices of the ``mother''-hypersimplex. 
The flags of matroid base polytopes have been studied in the literature, 
in particular in~\cite{Sangwook-conf} and~\cite{Sangwook-JCT}, in which 
a characterization of the faces of the matroid base polytopes is presented.
Also, a formula for the $\mathbf{c}\mathbf{d}$-index of rank-two matroid 
base polytope is presented, describing the number of their flags in the 
most compact way possible, from a linear relations perspective.
(v) Last but not least, this family forms a Minkowski basis for certain
generalized permutahedra of the form $P_{n-1}(\tilde{x})$ as defined 
and discussed in~\cite[p.~13]{Postnikov} in terms of non-negative
integer combination as Minkowski sums. Indeed, matroid base polytopes
form a subclass of the family of generalized permutahedra as shown
in~\cite{Ard-Bene-Doker} where some of the work from~\cite{Postnikov}
is generalized, especially the volume of a general matroid base polytope.
This will discussed in Section~\ref{sec:gen-perm}.

The motivation for this paper stems from an observation on
the enumeration of the flags of the standard permutahedron, 
presented in Proposition~\ref{prp:exp-gen-func} here below, 
which we now will parse through and discuss.
 
Recall that the {\em permutahedron $\Pi_{n-1}$} is defined
to be the convex hull of 
$\{ (\pi(1), \pi(2),\ldots,\pi(n)) \in {\reals}^n : \pi \in S_n\}$
where $S_n$ is the symmetric group of degree $n$.
The faces of $\Pi_{n-1}$ have a nice combinatorial description
as presented in Ziegler~\cite[p.~18]{Ziegler}: each $i$-dimensional 
face of $\Pi_{n-1}$ can
be presented as an ordered partition of the set $[n] = \{1,\ldots,n\}$ into
exactly $n-i$ distinct parts. In particular, $\Pi_{n-1}$ has
$\stirling{n}{n-i}(n-i)!$ faces of dimension $i$ for each 
$i\in\{0,1,\ldots,n\}$, where $\stirling{n}{k}$ denotes the Stirling 
number of the 2nd kind. 

{\sc Conventions:} (i) For an $\ell$-tuple 
$\tilde{x} = (x_1,\ldots,x_{\ell})$
of variables and an $\ell$-tuple of numbers $\tilde{a} = (a_1,\ldots,a_{\ell})$, 
let ${\tilde{x}}^{\tilde{a}} = x_1^{a_1}x_2^{a_2}\cdots x_{\ell}^{a_{\ell}}$.
(ii) For $\tilde{a} = (a_1,\ldots,a_{\ell})$ let
$\partial(\tilde{a}) = (a_1, a_2-a_1,a_3-a_2,\ldots,a_{\ell}-a_{\ell - 1})$.
The following definition is from \cite{MinkPoly}:
\begin{definition}
\label{def:flag-poly}
Let $P$ be a polytope with $\dim(P) = d$ and $\ell \in \nats$. 
For an $\ell$-tuple of variables $\tilde{x} = (x_1,\ldots,x_{\ell})$  
the {\em $\ell$-flag polynomial} of $P$ is defined by
\[
{\tilde{f}}^{\ell}_{P}(\tilde{x}) := 
\sum_{\tilde{s}}f_{\tilde{s}}(P){\tilde{x}}^{\partial({\tilde{s}})},
\]
where the sum is taken over all chains $\tilde{s} = (s_1,\ldots,s_{\ell})$
with $0\leq s_1\leq s_2\leq \cdots \leq s_{\ell}\leq d$
and $f_{\tilde{s}}(P)$ denotes the number of chains of faces
$A_1\subseteq A_2\subseteq \cdots \subseteq A_{\ell}$ 
of $P$ with $\dim(A_i) = s_i$ for each $i\in\{1,\ldots, \ell\}$.
\end{definition}
{\sc Conventions:}
(i) For a vector $\tilde{c} = (c_1,\ldots,c_n)$ we denote the
linear functional $\tilde{x}\mapsto \tilde{c}\cdot\tilde{x}$
by $L_{\tilde{c}}$.
(ii) For a given vector $\tilde{c}$ and a polytope $P$,
we denote by $F_P(\tilde{c})$ or just $F(\tilde{c})$ 
the unique face of $P$ determined
by $\tilde{c}$ as the points that maximize $L_{\tilde{c}}$ when 
restricted to $P$. Further, we denote the set of
all the faces of $P$ by $\mathbf{F}(P)$.
More specifically we denote the set of the $i$-dimensional
faces of $P$ by ${\mathbf{F}}_i(P)$, in particular, 
${\mathbf{F}}_0(P)$ denotes the set of vertices of $P$. 
(iii) Finally, for a vector $\tilde{c}$ 
the set $\supp(\tilde{c}) = \{c_1,\ldots,c_n\}$ is the {\em support}
of $\tilde{c}$.

Consider now the well-known description of the $i$-dimensional 
faces of $\Pi_{n-1}$ as the ordered partitions of $[n]$ into $n-i$ parts: 
more explicitly, each functional $L_{\tilde{c}}$ 
where the support $\supp(\tilde{c})$ has exactly $n-i$ distinct
values $c_1 < c_2 < \cdots < c_{n-i}$, when restricted to
$\Pi_{n-1}$, takes its maximum value at exactly one $i$-dimensional
face $A$. Here each value $c_i$ of the support corresponds uniquely
to one of the ordered parts defining the face $A$.
Also, by ``merging'' two such consecutive values 
$c_h$ and $c_{h+1}$ (for example, by replacing both $c_h$ and
$c_{h+1}$ by their average), we obtain a new functional
$L_{\tilde{c}'}$ which is maximized
at a face $A'$ of dimension $i+1$ that contains the face $A$.
So, by merging two consecutive parts into one part, we obtain
a coarser ordered partition of $[n]$. This merging process
can clearly be repeated. 
In this case we informally say that the first partition is a 
{\em refinement} of the last partition, or equivalently that the last partition 
is a {\em coarsening} of the first one. 
\begin{observation}
\label{obs:coarse}
For faces $A,B\in\mathbf{F}(\Pi_{n-1})$, then 
$A\subseteq B$ if, and only if, the ordered partition
of $[n]$ corresponding to $A$ is a refinement of
the ordered partition of $[n]$ corresponding to $B$.
\end{observation}
Given a chain $\tilde{s} = (s_1,\ldots,s_{\ell})$
with $0\leq s_1\leq s_2\leq \cdots \leq s_{\ell}\leq n-1$, 
the number $f_{\tilde{s}}(\Pi_{n-1})$ of chains of faces
$A_1\subseteq A_2\subseteq \cdots \subseteq A_{\ell}$ 
of $\Pi_{n-1}$ with $\dim(A_i) = s_i$ for each $i\in\{1,\ldots, \ell\}$
can then by Observation~\ref{obs:coarse} be obtained by 
first considering any of the 
$\stirling{n}{n-s_1}(n-s_1)!$ faces $A_1$ of dimension $s_1$, then
merging $s_2-s_1$ consecutive parts (in the ordered partition defining
the face $A_1$) of the $n-s_1-1$ available consecutive pairs,
then merging $s_3-s_2$ consecutive parts of the $n-s_2-1$ available
consecutive parts, and so on. Therefore 
the number  $f_{\tilde{s}}(\Pi_{n-1})$ of chains
of faces $A_1\subseteq A_2\subseteq \cdots \subseteq A_{\ell}$
where $\dim(A_i) = s_i$ for each $i$ is given by
\begin{eqnarray}
f_{\tilde{s}}(\Pi_{n-1}) & = & 
  \stirling{n}{n-s_1}(n-s_1)!
  \binom{n-s_1 -1}{s_2-s_1}\binom{n-s_2-1}{s_3-s_2}
  \cdots\binom{n-s_{\ell-1}-1}{s_{\ell}-s_{\ell-1}} \nonumber \\
 & = & \stirling{n}{n-s_1}(n-s_1)!
\binom{n-s_1 -1}{s_2-s_1 \ \cdots \ \ s_{\ell}-s_{\ell-1} \ \ n - s_{\ell}- 1}.
\label{eqn:chain-perm}
\end{eqnarray}
Such a simple formula for the number of $\tilde{s}$-chains of 
faces of $\Pi_{n-1}$
as in (\ref{eqn:chain-perm}) is not a coincidence, as it is solely 
the consequence of $\Pi_{n-1}$ being a simple polytope: each vertex 
of a simple $d$-polytope has $d$ neighboring vertices and is contained
in $d$ facets, and so each $k$-face containing a given vertex is uniquely
determined by $\binom{d}{k}$ of its neighbors. Hence, for each $h\leq k$
every $h$-face is contained in exactly $\binom{d-h}{k-h}$ $k$-faces,
and we obtain in general, as above, the following.
\begin{lemma}
\label{lmm:simple}
For any simple $d$-polytope $P$ and a chain $\tilde{s} = (s_1,\ldots,s_{\ell})$
with $0\leq s_1\leq s_2\leq \cdots \leq s_{\ell}\leq d$, we have 
\[
f_{\tilde{s}}(P) =
f_{s_1}(P) 
\binom{d-s_1}{s_2-s_1 \ \cdots \ \ s_{\ell}-s_{\ell-1} \ \ d - s_{\ell}},
\]
where $f_{s_1}(P)$ is the number of $s_1$-faces of $P$.
\end{lemma}
Now, assume for a moment that for a $d$-polytope $P$ we have
a polynomial ${\tilde{\phi}}^{\ell}_{P}$ of the form
\begin{equation}
\label{eqn:phi}
{\tilde{\phi}}^{\ell}_{P}(\tilde{x}) := 
\sum_{\tilde{s}}
\frac{f_{\tilde{s}}(P)}{D(d,s_1)}{\tilde{x}}^{\partial({\tilde{s}})},
\end{equation}
where $D$ is a bivariate function on non-negative integers. If
$D(d,s_1) = 1$ for all $d,s_1$ then 
${\tilde{\phi}}^{\ell} = {\tilde{f}}^{\ell}$, the $\ell$-flag polynomial from 
Definition~\ref{def:flag-poly}. If $P$ is simple, 
then by the above Lemma~\ref{lmm:simple}, the multinomial theorem and
the definition of the $f$-polynomial ($f = {\tilde{f}}^{1}$ obtained 
by letting $\ell = 1$ in Definition~\ref{def:flag-poly}), 
we obtain the following:
\begin{eqnarray*}
{\tilde{\phi}}^{\ell}_{P}(\tilde{x}) 
  & = & \sum_{\tilde{s}}\frac{f_{\tilde{s}}(P)}{D(d,s_1)}
{\tilde{x}}^{\partial({\tilde{s}})} \\
  & = & \sum_{\tilde{s}}
\frac{f_{s_1}(P)\binom{d-s_1}{s_2-s_1 \ \cdots \ \ s_{\ell}-s_{\ell-1} 
\ \ d - s_{\ell}}}{D(d,s_1)}
x_1^{s_1}x_2^{s_2-s_1}\cdots x_{\ell}^{s_{\ell}-s_{\ell-1}} \\ 
  & = & \sum_{s_1}\frac{f_{s_1}(P)}{D(d,s_1)}x_1^{s_1}\sum_{s_2,\ldots,s_{\ell}}
\binom{d-s_1}{s_2-s_1 \ \cdots \ \ s_{\ell}-s_{\ell-1} \ \ d - s_{\ell}}
x_2^{s_2-s_1}\cdots x_{\ell}^{s_{\ell}-s_{\ell-1}} \\
  & = & \sum_{s_1}\frac{f_{s_1}(P)}{D(d,s_1)}
x_1^{s_1}(x_2+\cdots+x_{\ell}+1)^{d-s_1} \\
  & = & (x_2+\cdots+x_{\ell}+1)^d\sum_{s_1}
\frac{f_{s_1}(P)}{D(d,s_1)}
\left(\frac{x_1}{x_2+\cdots+x_{\ell}+1}\right)^{s_1} \\ 
  & = & (x_2+\cdots+x_{\ell}+1)^d{\tilde{\phi}}^{1}_P
\left(\frac{x_1}{x_2+\cdots+x_{\ell}+1}\right),
\end{eqnarray*} 
showing that ${\tilde{\phi}}^{\ell}_{P}$ is uniquely determined by 
${\tilde{\phi}}^{1}_P$ if $P$ is a simple $d$ polytope. 
\begin{corollary}
\label{cor:simple-flag}
For a simple $d$-polytope $P$ and ${\tilde{\phi}}^{\ell}_{P}$ 
from (\ref{eqn:phi}) we have
\[
{\tilde{\phi}}^{\ell}_{P}(\tilde{x}) = 
(x_2 + \cdots + x_{\ell} + 1)^d
{\tilde{\phi}}^1_P\left(\frac{x_1}{x_2 + \cdots + x_{\ell} + 1}\right).
\]
\end{corollary}
In particular, the $\ell$-flag polynomial for any simple polytope
is uniquely determined by its $f$-polynomial
\[
{\tilde{f}}^{\ell}_{P}(\tilde{x}) = 
(x_2 + \cdots + x_{\ell} + 1)^d 
f_P\left(\frac{x_1}{x_2 + \cdots + x_{\ell} + 1}\right).
\]
{\sc Remark:} Despite this enumerative bonus for simple (and dually
for simplicial) polytopes, the number of the flags do not yield much of the 
actual face lattice structure of simple or simplicial polytopes.

Going back to our motivating permutahedron $\Pi_{n-1}$ and its number
$f_i(\Pi_{n-1}) = \stirling{n}{n-i}(n-i)!$ of faces, we see that for $n\geq 1$
\[
\sum_{i=0}^{n-1}\frac{f_i(\Pi_{n-1})}{(n-i)!}x^{n-i} 
= \sum_{i=0}^{n-1}\stirling{n}{n-i}x^{n-i} 
= T_n(x),
\]
where $T_n(x)$ is the Touchard polynomial of degree $n$, a.k.a.~the 
Bell polynomial in one variable of degree $n$, as 
$T_n(x) = B_n(x,\ldots,x)$ where $B_n(x_1,\ldots,x_n)$
is the {\em complete Bell polynomial} of degree $n$ in $n$ variables
denoted by $\phi_n(x_1,\ldots,x_n)$ in~\cite[p.~263]{Bell}),
and we have the corresponding bivariate exponential generating 
function~\cite[p.~265]{Bell} 
\begin{equation}
\label{eqn:Touch-St2}
T(x,y) 
= \sum_{n\geq 0}T_n(x)\frac{y^n}{n!} 
= \sum_{n,k \geq 0}\stirling{n}{k}x^k\frac{y^n}{n!} = e^{x(e^y-1)}.
\end{equation}
This suggests an exponential version of the $\ell$-flag polynomial
from Definition~\ref{def:flag-poly}.
\begin{definition}
\label{def:flag-poly-func}
Let $P$ be a $d$-polytope and $\ell\in\nats$. 
For an $\ell$-tuple of variables $\tilde{x} = (x_1,\ldots,x_{\ell})$  
define the {\em exponential ${\ell}$-flag polynomial} of $P$ by
\[
{\tilde{\xi}}^{\ell}_{P}(\tilde{x}) := 
\sum_{\tilde{s}}
\frac{f_{\tilde{s}}(P)}{(d-s_1+1)!}{\tilde{x}}^{\partial({\tilde{s}})},
\]
where the sum is taken over all chains $\tilde{s} = (s_1,\ldots,s_{\ell})$
with $0\leq s_1\leq s_2\leq \cdots \leq s_{\ell}\leq d$.

For each $a\geq 0$ define the 
{\em exponential ${\ell}$-generating function} of a given
family of polytopes ${\cal{P}} = \{P_d\}_{d\geq 0}$, where each $P_d$ is
of dimension $d$, by
\[
{\tilde{\xi}}_{{\cal{P}};a}^{\ell}(\tilde{x},y) := 
\sum_{d\geq a}{\tilde{\xi}}^{\ell}_{P_d}(\tilde{x})\frac{y^{d+1}}{(d+1)!}.
\] 
\end{definition}
In the case of $\ell=1$ we call
${\xi}_{P}(x) := {\tilde{\xi}}_{P}^{1}(x)$ the 
{\em exponential face (or $f$-) polynomial} of $P$ and 
for a family of polytopes ${\cal{P}} = \{P_d\}_{d\geq 0}$, 
each $P_d$ a $d$-polytope, 
we call ${\xi}_{{\cal{P}};a}(x,y) := {\tilde{\xi}}_{{\cal{P}};a}^{1}(x,y)$ 
the {\em exponential face function} of ${\cal{P}}$. When there is
not ambiguity and both the family ${\cal{P}}$ and the starting point $a$ are
clear, we omit the subscript in ${\xi}_{{\cal{P}};a}$ and simply write
${\xi}$.

For $a=0$ and ${\cal{P}} = \{\Pi_{n-1}\}_{n\geq 1}$ we get by 
(\ref{eqn:Touch-St2})
\begin{eqnarray*}
{\xi}(x,y) & = & \sum_{n\geq 1}{\xi}_{\Pi_{n-1}}(x)\frac{y^{n}}{n!} \\
  & = & \sum_{n\geq 1}
\left(\sum_{i=0}^{n-1}\frac{f_i(\Pi_{n-1})}{(n-i)!}x^i\right)\frac{y^{n}}{n!} \\
  & = & \sum_{n\geq 1}
\left(\sum_{i=0}^{n-1}\frac{f_i(\Pi_{n-1})}{(n-i)!}x^{-(n-i)}\right)
\frac{(xy)^{n}}{n!} \\
  & = & \sum_{n\geq 1}T_n(x^{-1})\frac{(xy)^{n}}{n!} \\
  & = & T(x^{-1},xy) - 1. 
\end{eqnarray*}
So by Corollary~\ref{cor:simple-flag} applied to 
${\tilde{\xi}}^{\ell}_{\Pi_{n-1}}$ we then get
\[
{\tilde{\xi}}^{\ell}(\tilde{x},y)
= \frac{1}{x_2 + \cdots + x_{\ell} + 1}{\xi}
\left(\frac{x_1}{x_2 + \cdots + x_{\ell} + 1},
(x_2 + \cdots + x_{\ell} + 1)y\right)
= \frac{T\left(\frac{x_2 + \cdots + x_{\ell} + 1}{x_1},
x_1y\right)-1}{x_2 + \cdots + x_{\ell} + 1},
\]
and again by (\ref{eqn:Touch-St2}) we get the following proposition.
\begin{proposition}
\label{prp:exp-gen-func}
The exponential generating function for 
all the $\ell$-flags of all the permutahedra $\Pi_{n-1}$ for $n\geq 1$
from Definition~\ref{def:flag-poly-func} is given by
\[
{\tilde{\xi}}^{\ell}(\tilde{x},y) = 
\sum_{n\geq 1,\tilde{s}}{\tilde{\xi}}^{\ell}_{\Pi_{n-1}}(\tilde{x})\frac{y^n}{n!} = 
\frac{e^{\frac{x_2 + \cdots + x_{\ell} +1}{x_1}(e^{x_1y}-1)} - 1}
{x_2 + \cdots + x_{\ell} +1}.
\]
\end{proposition}
In particular, as the coefficient
$[{\tilde{x}}^{\partial({\tilde{s}})}y^n]{\tilde{\xi}}^{\ell}(\tilde{x},y)$ 
of ${\tilde{x}}^{\partial({\tilde{s}})}y^n$ in the expansion of 
${\tilde{\xi}}_e^{\ell}(\tilde{x},y)$ is given by
\[
[{\tilde{x}}^{\partial({\tilde{s}})}y^n]{\tilde{\xi}}_e^{\ell}(\tilde{x},y)=
\frac{f_{\tilde{s}}(\Pi_{n-1})}{(n-s_1)!n!},
\]
then by Proposition~\ref{prp:exp-gen-func} we have 
\[
f_{\tilde{s}}(\Pi_{n-1}) = (n-s_1)!n![{\tilde{x}}^{\partial({\tilde{s}})}y^n]
\left(
\frac{e^{\frac{x_2 + \cdots + x_{\ell} +1}{x_1}(e^{x_1y}-1)} - 1}
{x_2 + \cdots + x_{\ell} +1}
\right).
\]
{\sc Remarks:}  
(i) Needless to say, there are many ways to define
{\em an} exponential generating function for the ${\ell}$-flags of
the permutahedra; we chose one here that would yield nice formulae.
(ii) Note that for $\ell= 1$ in Proposition~\ref{prp:exp-gen-func}, the
sum $x_2 + \cdots + x_{\ell}$ is empty which yields the exponential 
face function 
\[
{\xi}(x,y) = e^{\left(\frac{e^{xy}-1}{x}\right)} - 1.
\]

Having presented our motivating example, a natural question arises 
whether formulae as in Proposition~\ref{prp:exp-gen-func} can be 
generalized to a larger family of polytopes that include the permutahedron
$\Pi_{n-1}$. This will be the subject of the rest of the paper, which
is organized as follows. In Section~\ref{sec:our-fam} we formally
define the polytopes $\Pi_{n-1}(k-1)$ for each $k\geq 1$ and $n\geq k$
and we present some basic properties. In Section~\ref{sec:gen-perm} we
describe how the {\em PI-family}
${\cal{P}}_n = \{\Pi_{n-1}(k-1)\}_{k = 2,\ldots,n}$
fits in with various other families that generalize the standard
permutahedron and we demonstrate how ${\cal{P}}_n$ forms a Minkowski
basis for one such family of polytopes. The remaining two sections form
the meat of this paper. In Section~\ref{sec:flag-poly} we derive
a formula for the $f$-polynomial of $\Pi_{n-1}(k-1)$ and describe
its flags in terms of ordered pseudo-partitions of $[n] = \{1,\ldots,n\}$,
in a similar way as we did in Observation~\ref{obs:coarse} for 
the standard permutahedron $\Pi_{n-1}$. Finally, in Section~\ref{sec:exp-func}
we derive a closed formula for the exponential ${\ell}$-generating function
${\tilde{\xi}}_{{\cal{P}}_k^{\perp};k-1}^{\ell}(\tilde{x},y)$ for an arbitrary but
fixed integer $k\geq 1$, where 
${\cal{P}}_k^{\perp} = \{\Pi_{n-1}(k-1)\}_{n\geq k}$. 
Note that both families ${\cal{P}}_n$ and ${\cal{P}}_k^{\perp}$ cover
all the polytopes $\Pi_{n-1}(k-1)$ when $n$ and $k$ roam respectively;
$\bigcup_{n\geq 2}{\cal{P}}_n = \bigcup_{k\geq 2}{\cal{P}}_k^{\perp}$
are both partitions of the set of all $\Pi_{n-1}(k-1)$.

\section{The PI-family of polytopes and basic properties}
\label{sec:our-fam}

In this section we define the PI-family of polytopes we investigate 
and present some basic properties that naturally
generalize those of the permutahedron $\Pi_{n-1}$ and the  simplex 
$\Delta_{n-1}$. First we recall some basic definitions and notations we
will be using.

For $n\in\nats$ and $[n] = \{1,2,\ldots,n\}$, the {\em (standard) simplex}
$\Delta_{n-1} = \Delta_{[n]}$ of dimension $n-1$ is given by
$
\Delta_{n-1} = \Delta_{[n]} = 
\{ \tilde{x} = (x_1,\ldots,x_n)\in {\reals}^n : x_i\geq 0 \mbox{ for all $i$ },
x_1+\cdots+x_n = 1\}.
$
Each subset $F\subseteq [n]$ yields a face 
$\Delta_F$ of $\Delta_{[n]}$ given by
$
\Delta_F = \{ \tilde{x} \in \Delta_{[n]} : x_i = 0 \mbox{ for } i\not\in F\}.
$
Clearly $\Delta_F$ is itself a simplex embedded in ${\reals}^n$. 
If ${\cal{F}}$ is a family of subsets of $[n]$, then we can form the 
{\em Minkowski sum} of simplices
\begin{equation}
\label{eqn:Minks-basic}
P_{\cal{F}} = \sum_{F\in{\cal{F}}}\Delta_F =
\left\{\sum_{F\in{\cal{F}}}\tilde{x}_F : 
\tilde{x}_F \in \Delta_F \mbox{ for each } F\in{\cal{F}}\right\}.
\end{equation}
In general, every nonempty face of any polytope $P\subseteq {\reals}^n$ 
(in particular of $\Delta_{[n]}$) 
is given by the set of points that maximize a linear
functional $L_{\tilde{c}} : \tilde{x}\mapsto\tilde{c}\cdot\tilde{x}$ 
restricted to $P$.
We note that the permutahedron $\Pi_{n-1}$ can be 
expressed as a {\em zonotope}, a Minkowski sum of 
simplices each of dimension one:
\[
\Pi_{n-1} = \sum_{F\subseteq [n],\ |F|=2}\Delta_F.
\]
In light of this we obtain a natural generalization
\begin{equation}
\label{eqn:family}
\Pi_{n-1}(k-1) := \sum_{F\subseteq [n],\ |F|=k}\Delta_F,
\end{equation}
for each fixed $k\geq 2$, the Minkowski sum of all
$(k-1)$-dimensional simplices in ${\reals}^n$. 
We will refer to $\Pi_{n-1}(k-1)$ from (\ref{eqn:family}) 
as the {\em general permutahedron}.
Note that $\Pi_{n-1}(1) = \Pi_{n-1}$, the standard permutahedron,
and $\Pi_{n-1}(n-1) = \Delta_{n-1}$, the standard $(n-1)$-dimensional simplex.

We now present some basic facts about Minkowski sums
of polytopes in general that we will be using.
\begin{lemma}
\label{lmm:sum-faces}
Let $P_1,\ldots, P_k$ be polytopes in ${\reals}^n$. 
Then $F\in \mathbf{F}\left(\sum_{i=1}^k P_i\right)$ 
iff 
(i) $F = \sum_{i=1}^k A_i$ where each $A_i\in\mathbf{F}(P_i)$ and
(ii) there is a linear functional $L$ on ${\reals}^n$ 
such that each $L|_{P_i}$ is maximized at $A_i$ and
$L|_{\sum_{k=1}^k P_i}$ is maximized at $F$.
\end{lemma}
\begin{proof}
Clearly $F\in \mathbf{F}\left(\sum_{i=1}^k P_i\right)$ 
iff there is a linear functional $L$ on ${\reals}^n$ such 
that $L|_{\sum_{i=1}^k P_i}$ is maximized at $F$.
Let each $L|_{P_i}$ be maximized at $A_i\in \mathbf{F}(P_i)$. 
For $\sum_{i=1}^k x_i\in \sum_{i=1}^k P_i$ we then have
\[
L|_{\sum_{i=1}^k P_i}\left(\sum_{i=1}^k x_i\right)
= \sum_{i=1} L|_{P_i}(x_i)
\leq \sum_{i=1}^k\max(L|_{P_i}),
\]
and equality holds iff 
$L|_{P_i}(x_i) = \max(L|_{P_i})$ for each $i$,
which holds iff $x_i\in A_i$ for each $i$.
By definition of $F$ and $L$ this implies
that $F = \sum_{i=1}^k A_i$. The converse is easier.
\end{proof}
Note that the proof implies that
$\max(L|_{\sum_{i=1}^k P_i}) = \sum_{i=1}^k\max(L|_{P_i})$.

Assume now that $F = \sum_{i=1}^k A_i = \sum_{i=1}^k A_i'$ where 
$A_i ,A_i'\in \mathbf{F}(P_i)$ for each $i$, 
and where both tuples $(A_1,\ldots,A_k)$ and $(A_1',\ldots,A_k')$
are defined by the same functional $L$ and $L'$ 
respectively as in Lemma~\ref{lmm:sum-faces}.
If $(A_1,\ldots,A_k) \neq (A_1',\ldots,A_k')$, then
we may WLOG assume that 
$A_1'\setminus A_1 \neq \emptyset$. Then for 
$x_1'\in A_1'\setminus A_1$ and $x_i'\in A_i'$ for each $i\geq 2$ we have
\[
L|_{\sum_{i=1}^k P_i}\left(\sum_{i=1}^k x_i'\right)
= \sum_{i=1}^k L|_{P_i}(x_i')
< \max(L|_{P_1}) + \sum_{i=2}^k L|_{P_i}(x_i')
\leq \sum_{i=1}^k\max(L|_{P_i}).
\]
contradicting the fact that $\sum_{i=1}^kx_i' \in\sum_{i=1}^k A_i$ since
$L|_{\sum_{i=1}^k P_i}$ is maximized at $\sum_{i=1}^k A_i$.
Hence, we must have  $(A_1,\ldots,A_k) = (A_1',\ldots,A_k')$.
\begin{corollary}
\label{cor:AB=unique}
The decomposition of $F\in \mathbf{F}\left(\sum_{i=1}^k P_i\right)$ 
as $F = \sum_{i=1}^k A_i$ from Lemma~\ref{lmm:sum-faces} is unique.
\end{corollary}

Consider for a moment a functional $L_{\tilde{c}}$ where $c_1 > \cdots > c_n$.
Clearly $L_{\tilde{c}}$ restricted to $\Pi_{n-1}(k-1)$ as defined
in (\ref{eqn:family}) will yield the vertex $(1,0,\ldots,0)$
of $\Delta_F$ from all the $\binom{n-1}{k-1}$ subsets $F\subseteq [n]$
with $1\in F$. Further, $L_{\tilde{c}}$ will yield the vertex
$(0,1,0,\ldots,0)$ of $\Delta_F$ from all the $\binom{n-2}{k-1}$
subsets of $F\subseteq [n]$ with $1\not\in F$ and $2\in F$ and so on.
Hence, by (\ref{eqn:family}) and Lemma~\ref{lmm:sum-faces} 
we see that $L_{\tilde{c}}$ will yield
the unique vertex $\left(\binom{n-1}{k-1},\binom{n-2}{k-2},\ldots,
\binom{k-1}{k-1},0,\ldots,0\right)$ of $\Pi_{n-1}(k-1)$. 
By considering all permutations on $n$ indices, we therefore have 
the following proposition.
\begin{proposition}
\label{prp:vert}
Every vertex of $\Pi_{n-1}(k-1)$ has form $\tilde{u} = (u_1,\ldots,u_n)$
where the support is given by
\[
\supp(\tilde{u}) = \{u_1,\ldots,u_n\} 
= \left\{\binom{n-1}{k-1},\binom{n-2}{k-1},\ldots,
\binom{k-1}{k-1},0\right\}.
\]
There are exactly $k-1$ copies of $0$ among $u_1,\ldots,u_n$ and hence exactly
one copy of each nonzero integer from the above set. In particular 
$\Pi_{n-1}(k-1)$ has exactly $\frac{n!}{(k-1)!}$ vertices. 
\end{proposition}
From Proposition~\ref{prp:vert} here above we see that
$\Pi_{n-1}(k-1)$ is a degenerate case of the polytope 
$P_{n-1}(\tilde{v})$ from Postnikov~\cite{Postnikov}, defined 
as the convex hull of $\{(v_{\pi(1)},v_{\pi(2)},\ldots,v_{\pi(n)}) : \pi\in S_n\}$
for a fixed vector $\tilde{v}\in{\reals}^n$.
Namely, $\Pi_{n-1}(k-1) = P_{n-1}(\tilde{v})$ where 
$\tilde{v} = (v_1,\ldots,v_n) =
(\binom{n-1}{k-1},\binom{n-2}{k-1},\ldots,\binom{k-1}{k-1},0,\ldots,0)$.
The combinatorial type of $P_n(\tilde{v})$, for any 
$\tilde{v}\in{\reals}^n$
with $v_1 = v_2 = \cdots = v_{k-1} < v_k < \cdots < v_n$, is the same as that
of $\Pi_{n-1}(k-1)$ (i.e.~they have isomorphic face lattices)
so, in particular, the combinatorial type of $P_n(\tilde{v})$ 
when all the $v_i$ are distinct, is the same as that of the
standard permutahedron $\Pi_{n-1}$. In~\cite{Postnikov} the volume
$P_n(\tilde{v})$ is studied extensively, and it is shown to
be a polynomial in the variables $v_1,\ldots,v_n$.

Note that every functional $L_{\tilde{c}}$ for which
the coordinates $c_1,\ldots,c_n$ of $\tilde{c}$ are distinct will always 
yield a vertex of $\Pi_{n-1}(k-1)$, but not vice versa when $k\geq 3$.
For such a $\tilde{c}$ we can, as right before Observation~\ref{obs:coarse},
``merge'' two consecutive values the support of $\tilde{c}$ 
(i.e.~replace both values by their average, say) and thereby 
obtain the unique edge of $\Pi_{n-1}(k-1)$, the endvertices
of which form the max-set of this altered $L_{\tilde{c}}$.
Note that the edges of $\Pi_{n-1}(k-1)$ are of two types or kinds:
1st kind having $k-1$ zeros among the coordinates of each generic
point on the edge (i.e.~edge points excluding the endpoints), 
and the 2nd kind with $k-2$ zeros among the coordinates of each
generic point on the edge. The number $e_1$ of edges of the 1st kind
is the same as the number of ordered partition of a chosen $(n-k+1)$-subset 
of $[n]$ into $n-k$ parts, and hence 
\[
e_1 = \binom{n}{n-k+1}\stirling{n-k+1}{n-k}(n-k)! 
= \binom{n}{n-k+1}\binom{n-k+1}{2}(n-k)!
= \frac{(n-k)n!}{2(k-1)!}.
\]
The number $e_2$ of edges of the 2nd kind is the same as the number ways to
choose an $(n-k+2)$-subset of $[n]$, partition it into $n-k+1$ parts and order
the $n-k$ singletons of those parts, and so
\[
e_2 = \binom{n}{n-k+2}\stirling{n-k+2}{n-k+1}(n-k)! 
= \binom{n}{n-k+2}\binom{n-k+2}{2}(n-k)!
= \frac{n!}{2(k-2)!}.
\]
Hence, the total number of edges is given by 
$e_1 + e_2 = \frac{(n-1)n!}{2(k-1)!}$. We summarize in the following.
\begin{proposition}
\label{prp:v-e-simple}
Every edge of $\Pi_{n-1}(k-1)$ is between a pair of vertices as given
in Proposition~\ref{prp:vert}
that differ in exactly two coordinates whose values are consecutive
in the support of the vertices. Consequently the edges are of
two types: (i) edges between two vertices, both with the same
$k-1$ zero coordinates, and (ii) edges between two vertices, both with
the same $k-2$ zero coordinates and one with its least nonzero entry
where the other vertex has a zero. In particular, the number of edges 
of $\Pi_{n-1}(k-1)$ is $\frac{(n-1)n!}{2(k-1)!}$ and so 
$\Pi_{n-1}(k-1)$ is a simple polytope for all $k$ and $n$.
\end{proposition}
By the above Proposition~\ref{prp:v-e-simple} every $\Pi_{n-1}(k-1)$ 
is simple, so the sequence
\[
\Delta_{n-1} = \Pi_{n-1}(n-1), \Pi_{n-1}(n-2), \ldots, \Pi_{n-1}(2), 
\Pi_{n-1}(1) = \Pi_{n-1},
\]
can be viewed as discrete  
transition of simple polytopes from the  simplex $\Delta_{n-1}$ to
the standard permutahedron $\Pi_{n-1}$, see Figure~\ref{fig:chain}.
This is our first main reason to focus our study on
the PI-family consisting of $\Pi_{n-1}(k-1)$ where $k=2,\ldots,n-1$.
\begin{figure}
\begin{center}
\begin{tikzpicture}[scale=1]
\node (Pi31) at (-4,0) {\includegraphics{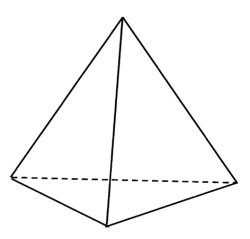}};
\node at (-4,-2) {$\Pi_3(3)$};
\node (Pi32) at (0,0) {\includegraphics[scale=.15]{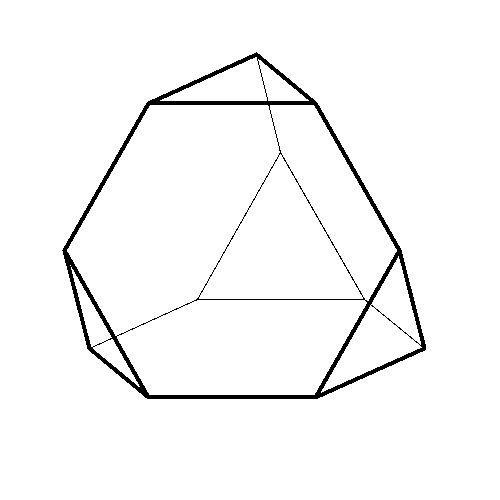}};
\node at (0,-2) {$\Pi_3(2)$};
\node (Pi33) at (4,0) {\includegraphics{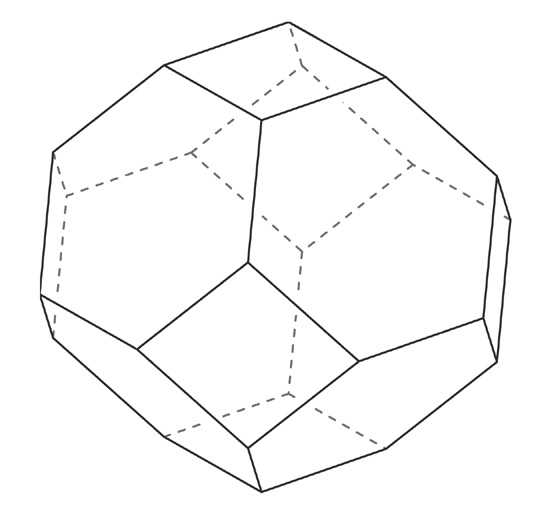}};
\node at (4,-2) {$\Pi_3(1)$};
\end{tikzpicture}
\end{center}
\caption{A simple transition between $\Delta_3$ and $\Pi_3$ in the case $n=4$.}
\label{fig:chain}
\end{figure}
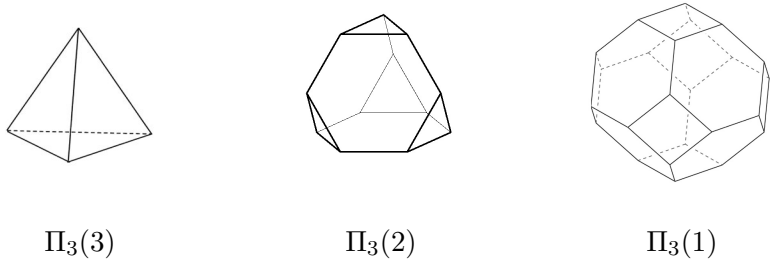
\section{Comparing various types of generalizations of  permutahedra}
\label{sec:gen-perm}

In this section we further promote the importance of the PI-family
${\cal{P}}_n = \{\Pi_{n-1}(k-1)\}_{k = 2,\ldots,n}$ and briefly 
compare various families of polytopes from
the literature, all generalizing the standard permutahedron in one 
form or another. We present some explicit characterizations
of them and show that the PI-family ${\cal{P}}_n$ forms a Minkowski
${\ints}^{+}$-basis for a large family of polytopes that generalizes
the standard permutahedron.

There are, needless to say, many ways to generalize the standard permutahedron
$\Pi_{n-1}$, and we have briefly mentioned two of them (namely, $\Pi_{n-1}(k-1)$
and $P_{n-1}(\tilde{v})$ from above). There are two other classes of 
important families of polytopes from~\cite{Postnikov} and 
from~\cite{Post-Vic-Laur} we want to relate $\Pi_{n-1}(k-1)$ to.
A good portion of the discussion immediately here below in this section,
is, in one form or another, contained in~\cite{Postnikov}
and~\cite{Rado}, except for some
minor observations, propositions, and examples toward the end of this section. 
We include it all in this short section though as it serves as a
second main reason for our investigation, as well as for 
self-containment of the article.

{\sc First Class:} 
For a collection $\tilde{Y} = \{y_I\}_{I\subseteq [n]}$ of non-negative 
real numbers $y_I\geq 0$ for each $I\subseteq [n]$, one can define a 
$P_{n-1}(\tilde{Y})$ as the Minkowski sum of  simplices 
$\Delta_I$ scaled by $y_I$
\[
P_{n-1}(\tilde{Y}) := \sum_{I\subseteq [n]}y_I\Delta_I,
\]
and is referred to as a {\em generalized permutohedron} in~\cite{Postnikov}.
Apriori this seems to be more general than the Minkowski sum in 
(\ref{eqn:Minks-basic}). However, if we consider a family ${\cal{F}}$ 
of subsets of $[n]$ containing
(possibly) multiple copies of subsets of $[n]$, then $P_{\cal{F}}$ 
from (\ref{eqn:Minks-basic}) can be written as 
\[
P_{\cal{F}} = \sum_{I\subseteq [n]}n_I\Delta_I,
\]
which can have the same combinatorial type as any $P_{n-1}(\tilde{Y})$. 
The class $P_{n-1}(\tilde{Y})$ (and hence also the class $P_{\cal{F}}$,) 
includes numerous classes of polytopes with
highly interesting combinatorial structures, like the {\em associahedron},
the {\em cyclohedron}, etc. (see~\cite{Postnikov} for many more examples.)  

{\sc Second Class:}
For a collection $\tilde{Z} = \{z_I\}_{I\subseteq [n]}$ of non-negative 
real numbers
$z_i\geq 0$ for each $I\subseteq [n]$, one can define $P_{n-1}(\tilde{Z})$
by its bounding hyperplanes
\[
P_{n-1}(\tilde{Z}) := \left\{ \tilde{x}\in{\reals}^n : 
\sum_{i\in[n]} x_i = z_{[n]}, \ \sum_{i\in I} x_i\leq z_I 
\mbox{ for }I\subset [n]\right\},
\]
which is also refereed to as the generalized permutohedron 
in~\cite{Postnikov}.
The following is a theorem of Rado~\cite{Rado}.
\begin{theorem}
\label{thm:Rado}
The polytope $P_{n-1}(\tilde{v})$ where
the coordinates are ordered $v_1\geq \cdots\geq v_n$, can be presented as those
$\tilde{t}\in{\reals}^n$ satisfying $\sum_{i\in[n]}t_i = \sum_{i\in[n]}v_i$ and
$\sum_{i\in I}t_i \leq \sum_{i\in[|I|]}v_i$ for each $I\subseteq [n]$.
\end{theorem}
From Theorem~\ref{thm:Rado} we see that if $z_I = z_J$ whenever $|I| = |J|$, 
then there are uniquely determined $v_1,\ldots,v_n$ such that 
$P_{n-1}(\tilde{Z}) = P_{n-1}(\tilde{v})$. Therefore, the class of 
$P_{n-1}(\tilde{Z})$ polytopes strictly includes all polytopes 
$P_{n-1}(\tilde{v})$.

Further, the following is from~\cite{Postnikov}.
\begin{proposition}
\label{prp:Postnikov}
For a given collection $\tilde{Y} = \{y_I\}_{I\subseteq [n]}$ of 
non-negative real numbers $y_I\geq 0$, then
$P_{n-1}(\tilde{Y}) = P_{n-1}(\tilde{Z})$ where $z_I = \sum_{J\subseteq I}y_J$
for each $I\subseteq [n]$. 
\end{proposition}
Hence, the the class of $P_{n-1}(\tilde{Z})$ polytopes 
also includes the class of all $P_{n-1}(\tilde{Y})$ polytopes. 
By Observation~\ref{obs:Z=Y} here below, this mentioned inclusion is 
also strict.

Suppose it is known that $P_{n-1}(\tilde{Z}) = P_{n-1}(\tilde{Y})$ for some 
$\tilde{Y} = \{y_I\}_{I\subseteq [n]}$. Then we may assume that 
$z_I = \sum_{J\subseteq I}y_J$ for each $I\subseteq [n]$. 
By a M\"{o}bius inversion we then get 
$y_I = \sum_{J\subseteq I}(-1)^{|I|-|J|}z_J$ for each $I\subseteq [n]$, so
the $y_I$ are uniquely determined in terms of the $z_I$. Hence we have 
the following.
\begin{observation}
\label{obs:Z=Y}
For $n\in\nats$ and a collection 
$\tilde{Z} = \{z_I\}_{I\subseteq [n]}$ of non-negative real numbers, 
then $P_{n-1}(\tilde{Z}) = P_{n-1}(\tilde{Y})$ if and only if 
$y_I = \sum_{J\subseteq I}(-1)^{|I|-|J|}z_J\geq 0$ for each $I\subseteq [n]$.
\end{observation}
For $n\in\nats$ we say that a 
collection $\tilde{T} = \{t_I\}_{I\subseteq[n]}$ is {\em symmetric} if 
$t_I = t_J$ whenever $|I|=|J|$. Hence, from Proposition~\ref{prp:Postnikov} 
and the above Observation~\ref{obs:Z=Y} we have 
the following.
\begin{observation}
\label{obs:symmetric}
For any $n\in\nats$ we have that
$\tilde{Y} = \{y_I\}_{I\subseteq [n]}$ is symmetric if and only if
$\tilde{Z} = \{z_I\}_{I\subseteq [n]}$ where $z_I = \sum_{J\subseteq I}y_J$
is symmetric. Further, if 
$\tilde{Y}$ is symmetric then
\[
P_{n-1}(\tilde{Y}) = \sum_{k=1}^ny_k\Pi_{n-1}(k-1)
\]
for non-negative real numbers $y_1,\ldots,y_k$,
where we interpret $\Pi_{n-1}(0)$ as a singleton point. 
\end{observation}
We now describe those permutahedra $P_{n-1}(\tilde{v})$ that can
be written as $P_{n-1}(\tilde{Y})$ for some 
$\tilde{y} = \{y_I\}_{I\subseteq [n]}$. 

For a sequence $(a_n)_{n\geq 0}$ of real numbers, recall the 
{\em (backward) difference} given by 
$\Delta(a_n) = a_n - a_{n-1}$ for each $n\geq 1$\footnote{The 
{\em forward difference} is defined as $\Delta(a_n) = a_{n+1}- a_n$}. 
Iteratively we also have the {\em $i$-th order difference} by
$\Delta^i(a_n) = \Delta(\Delta^{i-1}(a_n))$ for each $i\geq 0$ and
where $\Delta^0(a_n) = a_n$ for each $n$. Likewise,
for an $n$-tuple $\tilde{a} = (a_1,\ldots, a_n)\in {\reals}^n$ 
we let $\Delta(\tilde{a}) = (\Delta(a_2),\ldots,\Delta(a_n)) = 
(a_2-a_1, \ldots, a_n-a_{n-1})\in {\reals}^{n-1}$.
Clearly, if $P_{n-1}(\tilde{v})$ can be written of the form
$P_{n-1}(\tilde{Y})$ for some $\tilde{Y}$, then by Theorem~\ref{thm:Rado}
and Observation~\ref{obs:Z=Y} we can assume $\tilde{Y}$ to be 
symmetric and so
\[
P_{n-1}(\tilde{v}) 
= \sum_{I\subseteq [n]}y_{|I|}\Delta_I 
= y_1\tilde{1} + \sum_{k=2}^{n}y_k\Pi_{n-1}(k-1),
\]
where $\tilde{1} = (1,1,\ldots,1)\in {\reals}^n$. 
The following proposition is our main conclusion of this section,
the proof of which follows thereafter. 
\begin{proposition}
\label{prp:Y=X}
For $n\in\nats$ and $\tilde{v}\in {\reals}^n$ with $v_1\leq\cdots\leq v_n$, 
then $P_{n-1}(\tilde{v}) = \sum_{k=1}^{n}y_k\Pi_{n-1}(k-1)$ for non-negative
$y_k\geq 0$ for
each $k$, if and only if all the differences of $\tilde{v}$ are non-negative,
that is, $\Delta^i(v_k)\geq 0$
for each $i\in\{0,1\ldots,n-1\}$ and $k\in\{i+1,\ldots,n\}$.
\end{proposition}
{\sc Remark:} If $\Delta^i(v_k)\geq 0$ for each $i$ and $k$,
then coefficients $y_k$ with 
$P_{n-1}(\tilde{v}) = \sum_{k=1}^{n}y_k\Pi_{n-1}(k-1)$
are uniquely determined by $\tilde{v}$. Hence, 
$\{\Pi_{n-1}(k-1) : k\in \{1,\ldots,n\}\} = {\cal{P}}_n\cup\{\tilde{1}\}$
forms a {\em Minkowski basis} for such permutahedra $P_{n-1}(\tilde{v})$.

We will prove Proposition~\ref{prp:Y=X} in a few small steps. First a
lemma in linear algebra.
\begin{lemma}
\label{lmm:key}
For $n\in\nats$ and $k\in\{1,2,\ldots, n\}$ let
\[
\tilde{v}_{n-1}(k-1) 
= \left(0,\ldots,0,\binom{k-1}{k-1},\ldots,\binom{n-1}{k-1}\right).
\]
Then $\tilde{u}\in\spn_{{\reals}^{+}}(\{\tilde{v}_{n-1}(k-1) : 
k\in\{1,\ldots,n\}\})$
iff all the differences of $\tilde{u}$ are non-negative.
\end{lemma}
The proof of the above Lemma~\ref{lmm:key} will use the following
trivial fact.
\begin{claim}
\label{clm:trivial}
For any $n\in\nats$ and $\tilde{v}\in{\reals}^n$ then
$\Delta(\tilde{v}) = \tilde{0}\Leftrightarrow 
\tilde{v} = v\tilde{1} = (v,v,\ldots,v)\in {\reals}^{n-1}$.
\end{claim}
\begin{proof}(Lemma~\ref{lmm:key})
Note that $\Delta(\tilde{v}_{n-1}(k-1)) = \tilde{v}_{n-2}(k-2)$,
so by induction on $n$, all the differences of $\tilde{v}_{n-1}(k-1)$ 
are non-negative. Since the difference operator $\Delta$ is linear, 
then all the differences of $\sum_{k=1}^{n}y_k\tilde{v}_{n-1}(k-1)$
are non-negative if each $y_k\geq 0$. Therefore if
$\tilde{u}\in\spn_{{\reals}^{+}}(\{\tilde{v}_{n-1}(k-1) :
k\in\{1,\ldots,n\}\})$, then it is necessary for all the
differences of $\tilde{u}$ to be non-negative.

Conversely, let $\tilde{u}\in{\reals}^n$ have all its differences
non-negative. If $n=1$ then clearly 
$\tilde{u} = u_1 \in\spn_{{\reals}^{+}}(\{\tilde{v}_{0}(0)\})$. Otherwise
all the differences of $\Delta(\tilde{u})$ are non-negative, and hence
by induction on $n$ we can assume that 
$\Delta(\tilde{u}) = \sum_{k=2}^{n}y_{k}\tilde{v}_{n-2}(k-1)$
for some non-negative $y_2,\ldots, y_{n}$, and so 
\[
\Delta(\tilde{u}) 
= \sum_{k=2}^{n}y_{k}\Delta(\tilde{v}_{n-1}(k))
= \Delta\left(\sum_{k=2}^{n}y_k\tilde{v}_{n-1}(k-1)\right).
\]
By Claim~\ref{clm:trivial} we have
$\tilde{u} - \sum_{k=2}^{n}y_k\tilde{v}_{n-1}(k) = y_1\tilde{1}$
for some real $y_1$. Since all differences of $\tilde{u}$ are non-negative,
in particular $\Delta^0(\tilde{u}) = \tilde{u}$, we have $y_1 = u_1\geq 0$
and hence 
$\tilde{u}\in\spn_{{\reals}^{+}}(\{\tilde{v}_{n-1}(k-1) : k\in\{1,\ldots,n\}\})$.
\end{proof}
We now have what we need to prove Proposition~\ref{prp:Y=X}.
\begin{proof}(Proposition~\ref{prp:Y=X})
We first note
that if $\tilde{v}$ is as in Proposition~\ref{prp:Y=X}, that is
$P_{n-1}(\tilde{v}) = \sum_{k=1}^{n}y_k\Pi_{n-1}(k-1)$ where $y_k\geq 0$ for
each $k$, then $L_{\tilde{c}}$, where $\tilde{c} = (1,2,\ldots,n)\in{\reals}^n$,
is maximized at $\tilde{v}$, when restricted to $P_{n-1}(\tilde{v})$, and 
is maximized at $\tilde{v}_{n-1}(k-1)$
when restricted to $\Pi_{n-1}(k-1)$ for each $k$. Hence,
when restricted to $\sum_{k=1}^{n}y_k\Pi_{n-1}(k-1)$ then 
$L_{\tilde{c}}$ is maximized at $\sum_{k=1}^{n}y_k\tilde{v}_{n-1}(k-1)$, and so
$\tilde{v} = \sum_{k=1}^{n}y_k\tilde{v}_{n-1}(k-1)$.
By Lemma~\ref{lmm:key}, all the differences of $\tilde{v}$ must then
be non-negative.

For the converse, if all the difference of $\tilde{v}$ are non-negative,
then by Lemma~\ref{lmm:key} there are non-negative real coefficients 
$y_k\geq 0$ such that $\tilde{v} =\sum_{k=1}^{n}y_k\tilde{v}_{n-1}(k-1)$,
at which $L_{\tilde{c}}$ where $\tilde{c} = (1,2,\ldots,n)\in{\reals}^n$ 
when restricted to 
both $P_{n-1}(\tilde{v})$ and $\sum_{k=1}^{n}y_k\Pi_{n-1}(k-1)$ is maximized at.
Similarly, for any permutation $\pi$ of $\{1,\ldots,n\}$,
the linear functional $L_{\pi(\tilde{c})}$ where 
$\pi(\tilde{c}) = (\pi(1),\ldots,\pi(n))$ when restricted to both 
$P_{n-1}(\tilde{v})$ and $\sum_{k=1}^{n}y_k\Pi_{n-1}(k-1)$ 
is maximized at 
\[
\pi(\tilde{v}) = (v_{\pi(1)},\ldots,v_{\pi(n)}) =
\sum_{k=1}^{n}y_k\pi(\tilde{v}_{n-1}(k-1)).
\]
By definition of $P_{n-1}(\tilde{v})$ we therefore see that every vertex
of $P_{n-1}(\tilde{v})$ is also a vertex of
$\sum_{k=1}^{n}y_k\Pi_{n-1}(k-1)$. But since every vertex of
$\sum_{k=1}^{n}y_k\Pi_{n-1}(k-1)$ has by 
Corollary~\ref{cor:AB=unique}
the unique form $\sum_{k=1}^{n}y_k\tilde{w}_k$ where each $\tilde{w}_k$ is
a vertex of $\Pi_{n-1}(k-1)$, and each $\tilde{w}_k$ is the maximum set
of the functional $L_{\pi(\tilde{c})}$ when restricted to $\Pi_{n-1}(k-1)$,
then {\em every} vertex of $\sum_{k=1}^{n}y_k\Pi_{n-1}(k-1)$ is indeed the 
maximum set of some $L_{\pi(\tilde{c})}$. 
Hence, the the polytopes $P_{n-1}(\tilde{v})$ and 
$\sum_{k=1}^{n}y_k\Pi_{n-1}(k-1)$ have the same set of vertices, and
so must be the same polytope. This completes the proof.
\end{proof}
{\sc Example:}
Consider the polytope $P_3(0,1,2,2)$, and assume it can be 
written as $P_3(\tilde{Y})$ for some $\tilde{Y} = \{y_I\}_{I\subseteq [4]}$. 
By Theorem~\ref{thm:Rado} and Observation~\ref{obs:Z=Y}
we can assume $\tilde{Y}$ to be symmetric and so 
$P_3(0,1,2,2) = \sum_{k=1}^4 y_k\Pi_3(k-1)$.
Looking at the differences of $(0,1,2,2)$ we get
\[
\left[
\begin{array}{c}
\Delta^0 \\
\Delta^1 \\
\Delta^2 \\
\Delta^3
\end{array}
\right]
=
\left[
\begin{array}{ccccccc}
0 &   & 1 &   & 2 &   & 2  \\
  & 1 &   & 1 &   & 0 &    \\
  &   & 0 &   &-1 &   &    \\
  &   &   &-1 &   &   &     
\end{array}
\right]
\]
containing two negative entries in the differences of $\tilde{v} =
(0,1,2,2)$. By Proposition~\ref{prp:Y=X} $P_3(0,1,2,2)$ cannot be
written in the form of $P_3(0,1,2,2) = \sum_{k=1}^4 y_k\Pi_3(k-1)$.
However, $P_3(0,1,2,2)$ is still a symmetric polytope and
has dimension $3$ by Lemma~\ref{lmm:neq0} here below.

By Proposition~\ref{prp:Y=X} we have the following.
\begin{corollary}
\label{cor:Mink-basis}
The PI-family ${\cal{P}}_n = \{\Pi_{n-1}(k-1)\}_{k=2,\ldots,n}$ forms
a Minkowski ${\ints}^{+}$-basis for those polytopes $P_{n-1}(\tilde{v})$
that are of the form $P_{n-1}(\tilde{Y})$ for some family  
$\tilde{Y} = \{y_I\}_{I\subseteq [n]}$ of non-negative real numbers.
\end{corollary}

\section{The flag polynomial of the general permutahedron}
\label{sec:flag-poly}

Having briefly compared three types of polytopes, $P_{n-1}(\tilde{v})$,
$P_{n-1}(\tilde{Y})$, and $P_{n-1}(\tilde{Z})$, each of which
can be viewed as generalizations of the standard permutahedron,
we see that the polytopes $P_{n-1}(k-1)$ for $k\in\{1,\ldots,n\}$
form a Minkowski basis for those polytopes $P_{n-1}(\tilde{v})$
that can be expressed as $P_{n-1}(\tilde{Y})$. Hence, this can
be viewed as a further justification for studying them, and
so we will in this section focus on the PI-family
${\cal{P}}_n = \{\Pi_{n-1}(k-1)\}_{k=2,\ldots,n}$ 
for a given $n\in\nats$.
We will discuss the face lattice and its flag polynomial. 
Since many formal statements are the same for $\Pi_{n-1}(k-1)$ as with 
the more general $P_{n-1}(\tilde{v})$ and are, in fact, more transparent, 
we will consider the polytope $P_{n-1}(\tilde{v})$ in many cases, 
and then derive corollaries about $\Pi_{n-1}(k-1)$.

First, we will derive some facts from linear algebra that
will come in handy in this section. 

Consider two points $\tilde{a},\tilde{c}\in{\reals}^n$ where 
neither of them has all its coordinates the same. In this
case there is a proper partition $A\cup B = [n]$ such that
$c_i > c_j$ for all $(i,j)\in A\times B$. As neither $A$ nor
$B$ is empty, we cannot have $a_i= a_j$ for all $(i,j)\in A\times B$,
since that would imply $a_i=a_j$ for all $i,j\in[n]$. 
Hence, there is an $(i,j)\in A\times B$ with $a_i\neq a_j$. 
If $\tau = (i,j) \in S_n$ then 
\[
L_{\tilde{c}}(\tilde{a}) - L_{\tilde{c}}(\tau(\tilde{a})) = 
c_ia_i + c_ja_j - (c_ia_j + c_ja_i) = (c_i-c_j)(a_i -a_j) \neq 0.
\]
Hence, we have the following.
\begin{lemma}
\label{lmm:neq0}
Let $\tilde{a}, \tilde{c}\in {\reals}^n$, neither of which
have all its coordinates the same. Then there is a transposition
$\tau\in S_n$ such that 
$L_{\tilde{c}}(\tau(\tilde{a}))\neq L_{\tilde{c}}(\tilde{a})$.
\end{lemma}
In particular, for $\tilde{a}$ and $\tilde{c}$ as in Lemma~\ref{lmm:neq0}, 
$P_{n-1}(\tilde{a})\not\subseteq\ker(L_{\tilde{c}})$ and so 
$\dim(P_{n-1}(\tilde{a})) = n-1$. Now, since 
$\Pi_{n-1}(k-1) = P_{n-1}(\tilde{v}_{n-1}(k-1))$, where 
$\tilde{v}_{n-1}(k-1))
= \left(0,\ldots,0,\binom{k-1}{k-1},\ldots,\binom{n-1}{k-1}\right)$ 
is as in Lemma~\ref{lmm:key}, we then have the following.
\begin{corollary}
\label{cor:dim}
Let $\tilde{v}\in {\reals}^n$. Then
\[
\dim(P_{n-1}(\tilde{v})) = \left\{
\begin{array}{ll}
  0 & \mbox{ if } v_1 = \cdots = v_n, \\
 n-1 & \mbox{ otherwise. }
\end{array}
\right.
\]
In particular $\dim(\Pi_{n-1}(k-1)) = n-1$ for every $k\in\{2,\ldots,n\}$.
\end{corollary}
We now generalize Corollary~\ref{cor:dim} slightly.

As the symmetric group $S_n$ denotes the group of 
bijections $[n]\rightarrow [n]$, 
we can adopt the notation $S(X)$ for the group
of bijections $X\rightarrow X$, where $X$ is a given set.
With this convention $S_n = S([n])$ and clearly 
$S(X)\cong S_{|X|}$ for any finite set $X$. For any collection
$X_1,\ldots, X_k$ of disjoint subsets of $[n]$ we then have
the internal product $S(X_1)S(X_2)\cdots S(X_k)$, a subgroup of 
$S([n])$ which is isomorphic to the direct product 
$S_{n_1}\times S_{n_2}\times\cdots \times S_{n_k}$ where 
$|X_i| = n_i$. For a vector $\tilde{v}\in{\reals}^n$ and a
subset $X$ of $[n]$ we let 
$\proj_{X} : {\reals}^n \rightarrow {\reals}^{|X|}$ denote
the projection onto all the coordinate in $X$. If $X = \{i\}$ 
is a singleton set, we let $\proj_i = \proj_{\{i\}}$ be the projection
onto the $i$-th coordinate. Further 
we let $\delta_X(\tilde{v})$ denote the indicator function 
\begin{equation}
\label{eqn:indicator}
\delta_X(\tilde{v}) = \left\{
\begin{array}{ll}
  0 & \mbox{ if } |\supp(\proj_X(\tilde{v}))| = 1, \\
  1 & \mbox{ otherwise }.
\end{array}
\right.
\end{equation}
We now have by Corollary~\ref{cor:dim} the following more
general statement.
\begin{proposition}
\label{prp:dim-gen}
For disjoint subsets $X_1,\ldots,X_h$ of $[n]$ and $\tilde{v}\in {\reals}^n$
we have 
\[
\dim(\conv(\{\pi(\tilde{v}) : \pi\in S(X_1)\cdots S(X_h)\}))
= \sum_{i=1}^h \delta_{X_i}(\tilde{v})(|X_i| - 1).
\]
\end{proposition}
Note that the above Proposition~\ref{prp:dim-gen} holds
in particular for every partition $X_1\cup\cdots\cup X_h = [n]$
of $[n]$.

We seek to describe the face lattice of the polytope $P_{n-1}(\tilde{v})$
where $\tilde{v}$ has non-negative real entries,
in a similar fashion as was done when describing the faces
of the standard permutahedron $\Pi_{n-1}$ earlier, namely by considering
the max set of a linear functional restricted to the polytope.
In~\cite{Post-Vic-Laur} the combinatorial structure of classes
of polytopes that include those of $P_{n-1}(\tilde{v})$
is studied in great depth. In particular, the $f$-, $h$- and
$\gamma$-vectors of these classes of polytopes are studied.
Many explicit formulae for the $h$- and $\gamma$-vectors 
involving descent statistics of permutations are given.
Here we take a different OR-like (operations research)
approach, involving linear functionals, 
that more directly relates to the characterization
of the faces as presented in Observation~\ref{obs:coarse}.

We say that two vectors $\tilde{a}$ and $\tilde{c}$ have
the same {\em order type} if $a_i\leq a_j \Leftrightarrow c_i\leq c_j$
for all $i,j\in\{1,\ldots,n\}$. The order type defines an 
equivalence relation among vectors $\tilde{c}\in{\reals}^n$, 
and clearly all vectors
of the same type yield the same face of $P_{n-1}(\tilde{v})$,
as the set of maximum points of $L_{\tilde{c}}$ 
when restricted to $P_{n-1}(\tilde{v})$.
Denote by $[\tilde{c}]$ the order type equivalence
class of the vector $\tilde{c}\in{\reals}^n$. So, if $F(\tilde{c})$
denotes the unique face as the set of maximum points of $L_{\tilde{c}}$ 
restricted to $P_{n-1}(\tilde{v})$, then, by the above, 
$F(\tilde{c}) = F(\tilde{c}')$ whenever $[\tilde{c}] = [\tilde{c}']$,
and hence the face $F([\tilde{c}])$ is well defined. 
Also note that $P_{n-1}(\tilde{v})$ and $P_{n-1}(\tilde{v}')$ 
have the same combinatorial type iff $[\tilde{v}] = [\tilde{v}']$.
Finally, if $\delta_X(\tilde{c})$ denote the indicator function 
from (\ref{eqn:indicator}), then clearly $\delta_X([\tilde{c}])$ 
is well defined. 

As real addition is commutative, then 
for any permutation $\pi\in S_n$ we have
\[
L_{\pi(\tilde{c})}(\pi(\tilde{x})) 
= \pi(\tilde{c})\cdot\pi(\tilde{x}) 
= \tilde{c}\cdot\tilde{x}
= L_{\tilde{c}}(\tilde{x}).
\]
Hence, if let $\pi(F) = \{\pi(\tilde{x}) : \tilde{x}\in F\}$,
then clearly we have the following.
\begin{observation}
\label{obs:perm-face}
For any permutation $\pi\in S_n$ we have
$\pi(F([\tilde{a}])) = F([\pi(\tilde{a})])$,
and, in particular, $P_{n-1}(\tilde{v}) = P_{n-1}(\pi(\tilde{v}))$.
\end{observation}
Consider the polytope $P_{n-1}(\tilde{v})$ for a given vector
$\tilde{v}$ with non-negative real entries. To describe the 
face $F(\tilde{c})$ of $P_{n-1}(\tilde{v})$, we first note
that $\tilde{c}$ yields a unique ordered partition of $[n]$
\begin{equation}
\label{eqn:part}
[n] = X_1(\tilde{c})\cup\cdots\cup X_h(\tilde{c}),
\end{equation}
where $c_i=c_j$ for all $i,j\in X_{\ell}(\tilde{c})$ and
$c_i < c_j$ if $i\in X_{\ell}(\tilde{c})$ and $j\in X_{\ell'}(\tilde{c})$
where $\ell < \ell'$. Note that $L_{\tilde{c}}$ restricted to 
the set of vertices ${\mathbf{F}}_0(P_{n-1}(\tilde{v}))$
takes its maximum value on those vertices $\tilde{u}$, 
the order of whose entries 
are in agreement with the order of the entries of $\tilde{c}$,
that is $c_i < c_j\Rightarrow u_i\leq u_j$. This is clearly a necessary
and sufficient
condition $\tilde{u}\in {\mathbf{F}}_0(P_{n-1}(\tilde{v}))$ must satisfy
in order for $L_{\tilde{c}}(\tilde{u})$ to be a maximum value of $L_{\tilde{c}}$
when restricted to $P_{n-1}(\tilde{v})$. 
Formally we have a following description.
\begin{observation}
\label{obs:c-face}
For a given $\tilde{c}\in{\reals}^n$ the face of $P_{n-1}(\tilde{v})$ determined
by $[\tilde{c}]$ is given by
\begin{eqnarray*}
F([\tilde{c}]) 
& = & 
\conv(\{\tilde{u}\in {\mathbf{F}}_0(P_{n-1}(\tilde{v})) : 
c_i < c_j\Rightarrow u_i\leq u_j \}) \\
& = & 
\conv(\{\tilde{u}\in {\mathbf{F}}_0(P_{n-1}(\tilde{v})) : 
u_i < u_j\Rightarrow c_i\leq c_j \}).
\end{eqnarray*}
\end{observation}
Clearly by Observation~\ref{obs:perm-face}, we can assume $\tilde{v}$
to be ordered in any way convenient for our purposes. In particular,
when describing the face $F([\tilde{c}])$ of $P_{n-1}(\tilde{v})$,
we can for simplicity assume that the order of $\tilde{v}$ agrees with 
that of $\tilde{c}$, so $v_i\leq v_j$ whenever $c_i < c_j$, that is
we can assume $\tilde{v} \in F(\tilde{c})$ by Observation~\ref{obs:c-face}.
In terms of the partition from (\ref{eqn:part}), we then
obtain another equivalent form by Proposition~\ref{prp:dim-gen}.  
\begin{proposition}
\label{prp:face-describe}
For a given $\tilde{c}\in{\reals}^n$ the face of $P_{n-1}(\tilde{v})$ 
determined by $[\tilde{c}]$ that contains $\tilde{v}$ is given by
\[
F([\tilde{c}]) = 
\conv(\{
\mu_1\cdots\mu_h(\tilde{v}) : \mu_i\in S(X_i(\tilde{c}))\subseteq S([n]), \ \
i\in \{1,\ldots,h\}
\}).
\]
In particular we have 
\[
\dim(F([\tilde{c}])) = 
\sum_{i=1}^h \delta_{X_i(\tilde{c})}(\tilde{v})(|X_i(\tilde{c})| - 1).
\]
\end{proposition}
We note that if $\tilde{c}$ and $\tilde{v}$ are both 
ordered, $c_1\leq\cdots\leq c_n$ and $v_1\leq\cdots\leq v_n$, and
$\delta_{X_{\ell}(\tilde{c})\cup X_{\ell + 1}(\tilde{c})}(\tilde{v}) = 0$
for some $\ell$, then we can replace each $c_i$ where 
$i\in X_{\ell}(\tilde{c})$ and $c_j$ where $j\in X_{\ell+1}(\tilde{c})$
with a single value between $c_i$ and $c_j$, say $(c_i+c_j)/2$, 
and thereby obtaining a vector $\tilde{c}'$ with a strictly
smaller support than $\tilde{c}$ such that $F([\tilde{c}]) = F([\tilde{c}'])$.
In this case we have merged the two consecutive intervals 
$X_{\ell}(\tilde{c})$ and $X_{\ell +1}(\tilde{c})$ into one interval
without altering the corresponding face of $P_{n-1}(\tilde{v})$
that these vectors determine.
\begin{definition}
\label{def:v-reduced}
For ordered vectors $\tilde{c},\tilde{v}\in {\reals}^n$ we say
that $\tilde{c}$ is {\em $\tilde{v}$-reduced} if 
for every $\ell\in\{1,\ldots, h\}$ from (\ref{eqn:part}) we have 
$\delta_{X_{\ell}(\tilde{c})\cup X_{\ell + 1}(\tilde{c})}(\tilde{v}) = 1$.
\end{definition}

Turning our attention now back to the more specific PI-family 
${\cal{P}}_n = \{\Pi_{n-1}(k-1)\}_{k=2,\ldots,n}$ we note that
vectors of distinct order type can yield the same face of $\Pi_{n-1}(k-1)$ 
when $k\geq 3$, but for $k=2$ (when $\Pi_{n-1}(k-1) = \Pi_{n-1}$, 
the standard permutahedron)
then each face corresponds uniquely to the order type of the vector
yielding it. 
\begin{observation}
\label{obs:vec-type-face}
For every $k\geq 2$ the map 
$[\tilde{c}]\mapsto F([\tilde{c}])\in {\mathbf{F}}(\Pi_{n-1}(k-1))$ 
is always surjective, and it is injective (and hence bijective) iff $k=2$.
In particular, the total number of order types $[\tilde{c}]$ where
$\tilde{c}\in{\reals}^n$ is the same as $|{\mathbf{F}}(\Pi_{n-1})|$, 
the total number of faces (including the polytope itself) of $\Pi_{n-1}$. 
\end{observation}
Since $\Pi_{n-1}(k-1) = P_{n-1}(\tilde{v}_{n-1}(k-1))$
where $\tilde{v}_{n-1}(k-1) 
= \left(0,\ldots,0,\binom{k-1}{k-1},\ldots,\binom{n-1}{k-1}\right)$
from Lemma~\ref{lmm:key}, then when considering a face 
$F([\tilde{c}])$ of $\Pi_{n-1}(k-1)$
we can assume $\tilde{c}$ to be $\tilde{v}_{n-1}(k-1)$-reduced.
Therefore we can assume the partition (or rather the disjoint union) 
of $[n]$ induced by
$\tilde{c}$ from (\ref{eqn:part}) to have the form
$[n] = Z(\tilde{c})\cup X_0(\tilde{c})\cup\cdots\cup X_p(\tilde{c})$ 
where $Z(\tilde{c})$ consists of those indices from $[n]$ whose
coordinates of $F(\tilde{c})$ all are zero, which could potentially
be empty. In fact, letting $n_{\ell} = |X_{\ell}(\tilde{c})|$, 
we see that $F(\tilde{c})$ is the unique face that is
the convex combination of those vertices of $\Pi_{n-1}(k-1)$ 
where the $n_p$ largest entries occur in coordinates from
$X_p(\tilde{c})$, the next largest entries occur in coordinates from
$X_{p-1}(\tilde{c})$, etc., the $n_0$ 2nd smallest entries, not all zero
(but where some could be zero), occur in coordinates from 
$X_0(\tilde{c})$ and lastly the smallest
entries, all zero, occur in coordinates from $Z(\tilde{c})$.
As noted, with this setup $Z(\tilde{c})$ could be empty.
We therefore must relax the notion
of partition in order to obtain a description of the face
$F(\tilde{c}) = F([\tilde{c}])$. 
\begin{definition}
\label{def:OPP}
For $n\in\nats$ call a tuple $(Z,X_0,\ldots, X_{p})$
an {\em ordered pseudo-partition} of $[n]$ (or an {\em OPP} for short)  
if $Z\cup X_0\cup\cdots\cup X_{p} = [n]$ is a disjoint union and
either $(X_0,\ldots, X_{p})$, or $(Z,X_0,\ldots, X_{p})$
is an ordered (proper) partition of $[n]$.
\end{definition}
{\sc Remark:} Although the above Definition~\ref{def:OPP} is motivated by a 
vector in $\tilde{c}\in {\reals}^n$, and its dot-product with
a vertex $\tilde{v}\in\Pi_{n-1}(k-1)$, the definition of an OPP does not 
depend on it. 
\begin{theorem}
\label{thm:face-gen-perm}
For $n\in\nats$ and $k\in\{1,\ldots,n\}$, then every $d$-face
of $\Pi_{n-1}(k-1)$ is in one-to-one correspondence with
an OPP $(Z,X_0,\ldots,X_{p})$ of $[n]$ where 
(i) $0\leq |Z| \leq k-1$, 
(ii) $k\leq |Z| + |X_0| \leq n$, and
(iii) $n - |Z| - p - 1 = d$.
\end{theorem}
\begin{proof}
From an OPP ${\cal{P}} = (Z,X_0,\ldots,X_{p})$ of $[n]$
satisfying the conditions (i) -- (iii) in Theorem~\ref{thm:face-gen-perm}
above, we obtain a vector $\tilde{c} = \tilde{c}({\cal{P}})$ with 
$\proj_l(\tilde{c}) = 0$ if $l\in Z$ and 
$\proj_l(\tilde{c}) = i+1$ if $l\in X_i$.
In this case the face $F(\tilde{c})$ is exactly 
the convex combination of those vertices of $\Pi_{n-1}(k-1)$ where
the largest $|X_{p}|$ entries occur in coordinates from $X_{p}$,
the largest $|X_{p-1}|$ entries of the remaining $n-|X_{p}|$ ones occur
in coordinates from $X_{p-1}$ etc, the largest $|X_0|$ entries of the 
remaining $n-(|X_1|+\cdots +|X_{p}|)$ ones occur in coordinates
from $X_0$, and finally, all the coordinates from $Z$ contain 
only zeros. Hence, each OPP ${\cal{P}}$ yields a unique 
face $F(\tilde{c}({\cal{P}}))$.

On the other hand, every (proper) face $F$ of $\Pi_{n-1}(k-1)$
has the form $F([\tilde{c}])$
for some $\tilde{c}\in{\reals}^n$ where $\supp(\tilde{c})=[h]$ 
for some $h\in[n]$ Viewing $\tilde{c}$ as a function $c : [n]\rightarrow [h]$
with $c(i) = c_i$ for each $i\in [n]$, we obtain a OPP 
${\cal{P}} = {\cal{P}}(\tilde{c})$ as in the following way.

Letting $p\geq 0$ be the least integer with
$|X_0| + \cdots + |X_{p}| \geq n-k+1$, where $X_i = c^{-1}(h-p+i)$
for each $i=0,\ldots,p$, and $Z = c^{-1}([h-p-1])$, 
will give us our desired OPP ${\cal{P}}(\tilde{c}) = (Z,X_0,\ldots,X_{p})$.

Clearly we have ${\cal{P}}(F(\tilde{c}({\cal{P}}))) = {\cal{P}}$ and
$F(\tilde{c}({\cal{P}}(F(\tilde{c})))) = F(\tilde{c})$. 
Note that in general $\tilde{c}({\cal{P}}(F(\tilde{c})))\neq\tilde{c}$,
but they yield the same face. This proves the one-to-one
correspondence between OPPs and (proper) faces of $\Pi_{n-1}(k-1)$.

Finally, by Proposition~\ref{prp:face-describe}
if $F = F(\tilde{c}({\cal{P}}))$
is the unique face obtained from the OPP ${\cal{P}}$, then 
\[
\dim(F) = \sum_{i=0}^{p}(|X_i| - 1) = 
|X_0| + \cdots + |X_{p}| - p - 1 =
n - |Z| - p - 1,
\]
which completes the proof.
\end{proof}
By Theorem~\ref{thm:face-gen-perm} we can derive the $f$-polynomial
of $\Pi_{n-1}(k-1)$ by enumerating all OPP ${\cal{P}}$ satisfying
(i) and (ii) in Theorem~\ref{thm:face-gen-perm} 
with $d = n - |Z| - p - 1$ being a given fixed number.
For disjoint $Z, X_0\subseteq [n]$ there are $\stirling{n-|Z|-|X_0|}{p}p!$
ordered partitions $(X_1,\ldots,X_p)$ of the remaining elements
of $[n]\setminus(Z\cup X_0)$.

Letting $i = |Z| \in \{0,\ldots,k-1\}$ and $j = |X_0|$, we 
get by Theorem~\ref{thm:face-gen-perm} that $i\in\{0,\ldots,k-1\}$
and $i+j\in \{k,\ldots,n\}$.
Hence, each ordered $(X_1,\ldots,X_p)$ of the remaining $n-i-j$
elements from $[n]\setminus(Z\cup X_0)$ will by 
Theorem~\ref{thm:face-gen-perm} yield a face of dimension $n - i - p - 1$.
As there are $\binom{n}{i}\binom{n-i}{j}$ ways of choosing a legitimate
pair $(Z,X_0)$, we have the following Proposition.
\begin{proposition}
\label{prp:face-poly}
The $f$-polynomial $f_{\Pi_{n-1}(k-1)}(x) = \sum_{i=0}^{n-1}f_i(\Pi_{n-1}(k-1))x^i$
of $\Pi_{n-1}(k-1)$ is given by 
\[
f_{\Pi_{n-1}(k-1)}(x) = \sum_{\substack{0\leq i\leq k-1 \\ k\leq i+j\leq n}}
\binom{n}{i}\binom{n-i}{j}\sum_{p=0}^{n-i-j}\stirling{n-i-j}{p}{p}!x^{n-i-p-1}.
\]
\end{proposition}
{\sc Remarks:}
(i) Note that the coefficients $[x^0]f_{\Pi_{n-1}(k-1)}(x)$ 
and $[x^1]f_{\Pi_{n-1}(k-1)}(x)$ agree with previous 
Propositions~\ref{prp:vert} and~\ref{prp:v-e-simple} on the number of
vertices and edges respectively. 
(ii) When $k=n$ we obtain 
\[
f_{\Pi_{n-1}(k-1)}(x) = f_{\Pi_{n-1}(n-1)}(x) = \frac{(x+1)^n-1}{x},
\]
the $f$-polynomial of the  $(n-1)$-dimensional simplex.

By Propositions~\ref{prp:face-poly} and~\ref{prp:v-e-simple}
and Corollary~\ref{cor:simple-flag} we obtain the $\ell$-flag polynomial 
of $\Pi_{n-1}(k-1)$ in the following.
\begin{corollary}
\label{cor:1-2-ell}
For each $k\in\{2,\ldots,n\}$ the $\ell$-flag polynomial 
${\tilde{f}}^{\ell}_{\Pi_{n-1}(k-1)}(\tilde{x})$ of $\Pi_{n-1}(k-1)$
is given by 
\[
{\tilde{f}}^{\ell}_{\Pi_{n-1}(k-1)}(\tilde{x}) = 
(x_2 + \cdots + x_{\ell} + 1)^{n-1}
f_{\Pi_{n-1}(k-1)}\left(\frac{x_1}{x_2+\cdots+x_{\ell}+1}\right),
\]
where $f_{\Pi_{n-1}(k-1)}$ is the $f$-polynomial of $\Pi_{n-1}(k-1)$ 
given in Proposition~\ref{prp:face-poly}.
\end{corollary}

We complete this section on the face lattice of $\Pi_{n-1}(k-1)$ 
by describing  the faces of $\Pi_{n-1}(k-1)$ in terms of OPPs 
of $[n]$ and when one face contains another in a similar 
fashion as in Observation~\ref{obs:coarse}.

From Observation~\ref{obs:c-face} we immediately obtain the
following.
\begin{proposition}
\label{prp:face-chain}
If $\tilde{a},\tilde{c}\in{\reals}^n$ are such
that $a_i\leq a_j\Rightarrow c_i\leq c_j$, then
for the corresponding faces of $P_{n-1}(\tilde{v})$ we have
$F([\tilde{a}])\subseteq F([\tilde{c}])$.
\end{proposition}
Proposition~\ref{prp:face-chain} yields
a sufficient condition for the vectors $\tilde{a}$ and $\tilde{c}$
that implies $F([\tilde{a}])\subseteq F([\tilde{c}])$.
We will now describe exactly the relationship between 
$\tilde{a}$ and $\tilde{c}$ such that 
for faces of $P_{n-1}(\tilde{v})$ we have
$F([\tilde{a}])\subseteq F([\tilde{c}])$.

Assume $\tilde{a},\tilde{c}\in {\reals}^n$ are such for
their corresponding faces of $P_{n-1}(\tilde{v})$ we have
that $F([\tilde{a}])\subseteq F([\tilde{c}])$. Since there
is a permutation $\alpha\in S_n$ with $\alpha(\tilde{a})$ ordered,
i.e.~$a_{\alpha(1)}\leq\cdots\leq a_{\alpha(n)}$, we can, for simplicity, 
assume $\tilde{a}$ is ordered $a_1\leq\cdots\leq a_n$. 
In this case the partition 
$[n] = X_1(\tilde{a})\cup\cdots\cup X_h(\tilde{a})$
of $[n]$ induced by $\tilde{a}$ as in (\ref{eqn:part})
is a union of consecutive intervals. By 
Observation~\ref{obs:perm-face} we can assume $\tilde{v}$
is ordered in the same way as $\tilde{a}$ is, so
$v_1\leq\cdots\leq v_n$. If $\delta_{X_{\ell}(\tilde{a})}(\tilde{v}) = 1$,
then for any transposition $\tau\in S(X_{\ell}(\tilde{a}))\subseteq S_n$ 
we have $\tau(\tilde{v})\in F([\tilde{a}])\subseteq F([\tilde{c}])$
and hence $\tilde{c}\cdot\tilde{v} = \tilde{c}\cdot\tau(\tilde{v})$.
By Lemma~\ref{lmm:neq0} we must have the following.
\begin{claim}
\label{clm:c=const}
$c_i = c_j$ for all $i,j\in X_{\ell}(\tilde{a})$ with 
$\delta_{X_{\ell}(\tilde{a})}(\tilde{v}) = 1$.
\end{claim}
Assume now $i\in X_{\ell}(\tilde{a})$ and $j\in X_{\ell + 1}(\tilde{a})$
where $\tilde{a}$ is $\tilde{v}$-reduced.
In this case one of the following three conditions hold:
(i) $v_i < v_j$, 
(ii) $v_i = v_j$ and $\delta_{X_{\ell}(\tilde{a})}(\tilde{v}) = 1$, 
and hence there is an 
$i'\in  X_{\ell}(\tilde{a})$ with $v_{i'} < v_i = v_j$, or 
(iii) $v_i = v_j$ and $\delta_{X_{\ell+1}(\tilde{a})}(\tilde{v}) = 1$, 
and hence there is an 
$j'\in X_{\ell+1}(\tilde{a})$ with $v_i = v_j < v_{j'}$. 

In case (i) consider the transposition $\tau = (i,j)$. 
Since $\tilde{v}\in F([\tilde{a}])\subseteq F([\tilde{c}])$
we have $\tilde{c}\cdot\tilde{v} \geq \tilde{c}\cdot\tau(\tilde{v})$, and hence
$c_iv_i + c_jv_j \geq c_iv_j + c_jv_i$ or $(c_j - c_i)(v_j - v_i) \geq 0$.
Therefore $c_i \leq c_j$ must hold.

In case (ii) consider the transposition $\tau = (i',j)$. 
As in previous case we have
$\tilde{c}\cdot\tilde{v} \geq \tilde{c}\cdot\tau(\tilde{v})$, and hence
$c_{i'}v_{i'} + c_jv_j \geq c_{i'}v_j + c_jv_{i'}$ 
or $(c_j - c_{i'})(v_j - v_{i'}) \geq 0$.
Therefore $c_{i'} \leq c_j$ must hold, and so by Claim~\ref{clm:c=const}
$c_i = c_{i'} \leq c_j$ must hold.

Finally, in case (iii) consider the transposition $\tau = (i,j')$. 
As in previously we have
$\tilde{c}\cdot\tilde{v} \geq \tilde{c}\cdot\tau(\tilde{v})$, and hence
$c_iv_i + c_{j'}v_{j'} \geq c_iv_{j'} + c_{j'}v_i$ 
or $(c_{j'} - c_i)(v_{j'} - v_i) \geq 0$.
Therefore $c_i \leq c_{j'}$ must hold, and so by Claim~\ref{clm:c=const}
$c_i \leq c_{j'} = c_j$ must hold. Hence, we have the following.
\begin{claim}
\label{clm:c-ordered}
For a $\tilde{v}$-reduced $\tilde{a}$, 
if $i\in X_{\ell}(\tilde{a})$ and $j\in X_{\ell'}(\tilde{a})$ 
with $\ell < \ell'$, then $c_i\leq c_j$.
\end{claim}
By Observation~\ref{obs:perm-face} and the previous two 
Claims~\ref{clm:c=const} and~\ref{clm:c-ordered}, noting that
the ordering of both $\tilde{a}$ and $\tilde{v}$ was assumed
for the sake of argument, we have the following summarizing theorem.
\begin{theorem}
\label{thm:face-subset}
Let $\tilde{a},\tilde{c}\in {\reals}^n$ where $\tilde{a}$ is 
$\tilde{v}$-reduced, and $F([\tilde{a}]), F([\tilde{c}])$ be the corresponding
induced faces of $P_{n-1}(\tilde{v})$. Assume 
$\tilde{v} \in F([\tilde{a}])$ and let
$[n] = X_1(\tilde{a})\cup\cdots\cup X_h(\tilde{a})$ be the 
partition of $[n]$ induced by $\tilde{a}$ as in (\ref{eqn:part}).
With this setup we have $F([\tilde{a}])\subseteq F([\tilde{c}])$
if and only if we have the following.
\begin{enumerate}
  \item $a_i < a_j\Rightarrow c_i\leq c_j$.
  \item For every part $X_{\ell}(\tilde{a})$ with
$\delta_{X_{\ell}(\tilde{a})}(\tilde{v}) = 1$, we have 
$\delta_{X_{\ell}(\tilde{a})}(\tilde{c}) = 0$.
\end{enumerate}
\end{theorem}
Finally in this section, we further seek a description of 
the faces and flags of $\Pi_{n-1}(k-1)$ as 
described in Observation~\ref{obs:coarse} for the standard permutahedron
$\Pi_{n-1}$. To do so, we apply Theorem~\ref{thm:face-subset}
to describe when exactly one face $F([\tilde{a}])$ of $\Pi_{n-1}(k-1)$
is contained in another $F([\tilde{c}])$ 
in terms of the characterization given in Theorem~\ref{thm:face-gen-perm}.
We can assume both $\tilde{a}$ and $\tilde{c}$ to be 
$\tilde{v}_{n-1}(k-1)$-reduced.

Note that, trivially, if $|X_{\ell}(\tilde{a})|=1$, then
clearly $\delta_{X_{\ell}(\tilde{a})}(\tilde{c}) = 0$.
Assume that $F([\tilde{a}])$ and $F([\tilde{c}])$ correspond to the
OPPs $(Z,X_0,\ldots,X_{p})$ and $(Z',X'_0,\ldots,X'_{p'})$ 
of $[n]$ respectively, and that 
$\tilde{v}\in F([\tilde{a}])\subseteq F([\tilde{c}])$
is a vertex of $\Pi_{n-1}(k-1)$. By Theorems~\ref{thm:face-gen-perm} 
and~\ref{thm:face-subset} and the above
note, we have that $\delta_{X_i}(\tilde{c}) = 0$ 
for all  $i\in\{0,\ldots,p\}$, and hence $X_0\subseteq X_0'$, 
and further each part from 
$\{X_1,\ldots,X_{p}\}$ is contained in a part from $\{X'_0,\ldots,X'_{p'}\}$.
So, as a direct consequence of Theorem~\ref{thm:face-subset}
We have the following.
\begin{corollary}
\label{cor:face-subset-OPP}
For two faces $F([\tilde{a}])$ and $F([\tilde{c}])$ of $\Pi_{n-1}(k-1)$,
where both $\tilde{a}$ and $\tilde{c}$ are $\tilde{v}_{n-1}(k-1)$-reduced,
corresponding to the OPPs $(Z,X_0,\ldots,X_{p})$ and 
$(Z',X'_0,\ldots,X'_{p'})$ respectively, we have 
$F([\tilde{a}])\subseteq F([\tilde{a}])$ if, and only if,
the disjoint union $D\cup X_0\cup\cdots\cup X_{p}$ a refinement
of $X'_0\cup\cdots\cup X'_{p'}$ where $X_0\subseteq X_0'$ and
$D = X_0'\setminus X_0$ is the difference.
\end{corollary}
{\sc Remark:} Note that
$D\subseteq [n]\setminus(X_0\cup\cdots\cup X_p)\subseteq Z$.

Note that we have a well defined map
$\{0,1,\ldots,p\}\ni i\mapsto i'\in \{0,1,\ldots,p'\}$ where $i'$ 
is the unique index with $X_i\subseteq X'_{i'}$. In this way we have.
\begin{observation}
\label{obs:inc-surj}
The above map $i\mapsto i'$ is an increasing surjection.
\end{observation}

\section{A closed formula for the exponential generating function}
\label{sec:exp-func}

In this final section we derive a closed formula for the 
exponential ${\ell}$-generating function 
${\tilde{\xi}}_{{\cal{P}}_k^{\perp};k-1}^{\ell}(\tilde{x},y)$ from 
Definition~\ref{def:flag-poly-func} of the family 
${\cal{P}}_k^{\perp} = \{\Pi_{n-1}(k-1)\}_{n\geq k}$ 
of all the general permutahedra,
which we will henceforth denote by ${\tilde{\xi}}_k^{\ell}(\tilde{x},y)$,
analogous to the result of Proposition~\ref{prp:exp-gen-func}.
Unless otherwise stated $k\geq 1$ is an arbitrary but fixed
integer throughout.

If we let  
\[
{\tilde{g}}_n^{\ell}(\tilde{x}) 
:= \sum_{\tilde{s}}\frac{f_{\tilde{s}}(\Pi_{n-1}(k-1))}{(n-s_1)!}
x_1^{n-s_1}x_2^{s_2-s_1}\cdots x_{\ell}^{s_{\ell}-s_{\ell-1}},
\]
then by Definition~\ref{def:flag-poly-func} we have
\begin{equation}
\label{eqn:gn}
x_1^n{\tilde{g}}_n^{\ell}(x_1^{-1},x_2,\ldots,x_{\ell}) 
=  {\tilde{\xi}}_{\Pi_{n-1}(k-1)}^{\ell}(\tilde{x})
\end{equation}
and so for
\[
{\tilde{g}}^{\ell}(\tilde{x},y) 
:= \sum_{n\geq 1}{\tilde{g}}_n^{\ell}(\tilde{x})\frac{y^n}{n!}
= \sum_{n\geq k}{\tilde{g}}_n^{\ell}(\tilde{x})\frac{y^n}{n!}
\]
we have
\begin{equation}
\label{eqn:g-2-g-tilde}
{\tilde{g}}^{\ell}(x_1^{-1},x_2,\ldots, x_{\ell},x_1y) = 
\sum_{n\geq k}x_1^n{\tilde{g}}_n^{\ell}(x_1^{-1},x_2,\ldots,x_{\ell})\frac{y^n}{n!}
= {\tilde{\xi}}_k^{\ell}(\tilde{x},y).
\end{equation}
Hence, it suffices to obtain a closed formula for 
${\tilde{g}}^{\ell}(\tilde{x},y)$. Further we note that 
for $\ell = 1$ we obtain by (\ref{eqn:gn}) that
$x_1^n{\tilde{g}}_n^1(x_1^{-1}) = {\tilde{\xi}}^1(x_1)$ and hence
${\tilde{g}}_n^1(x_1) = x_1^n{\tilde{\xi}}^1(x_1^{-1})$ and so
by Proposition~\ref{prp:face-poly} that
\begin{equation}
\label{eqn:g}
g_n(x) := {\tilde{g}}_n^{1}(x) = 
\sum_{\substack{0\leq i\leq k-1 \\ k\leq i+j\leq n}}
\binom{n}{i}\binom{n-i}{j}\sum_{p=0}^{n-i-j}\stirling{n-i-j}{p}{p}!
\frac{x^{i+p+1}}{(i+p+1)!}.
\end{equation}
Further, by (\ref{eqn:gn}) for general $\ell$ and $\ell=1$, 
and Corollary~\ref{cor:simple-flag} applied to ${\tilde{\xi}}^{\ell}$ 
and (\ref{eqn:g}) we obtain
\begin{eqnarray*}
\lefteqn{{\tilde{g}}_n^{\ell}(\tilde{x}) = 
\frac{1}{x_2 + \cdots + x_{\ell} +1}g_n(x_1(x_2 + \cdots + x_{\ell}+1))} \\
   & & = \frac{1}{x_2 + \cdots + x_{\ell} +1}\left[
\sum_{\substack{0\leq i\leq k-1 \\ k\leq i+j\leq n}}
\binom{n}{i}\binom{n-i}{j}\sum_{p=0}^{n-i-j}\stirling{n-i-j}{p}{p}!
\frac{(x_1(x_2 + \cdots + x_{\ell}+1))^{i+p+1}}{(i+p+1)!}
\right].
\end{eqnarray*}
Hence, it suffices to obtain a closed formula for 
\begin{equation}
\label{eqn:gxy}
g(x,y) := {\tilde{g}}^1(x,y) 
= \sum_{n\geq 1}g_n(x)\frac{y^n}{n!} = \sum_{n\geq k}g_n(x)\frac{y^n}{n!},
\end{equation}
since then 
\begin{equation}
\label{eqn:g-tilde-y}
{\tilde{g}}_n^{\ell}(\tilde{x},y) = 
\sum_{n\geq k}{\tilde{g}}_n^{\ell}(\tilde{x})\frac{y^n}{n!} 
= \frac{g(x_1(x_2 + \cdots + x_{\ell}+1),y)}{x_2 + \cdots + x_{\ell} +1}.
\end{equation}

By (\ref{eqn:g}) we have 
\begin{equation}
\label{eqn:g=sumgi}
g_n(x) = \sum_{i=0}^{k-1}g_{n;i}(x)
\end{equation}
where for each $i\in\{0,1,\ldots,k-1\}$  
\begin{eqnarray}
g_{n;i}(x) & = & \sum_{j=k-i}^{n-i}
\binom{n}{i}\binom{n-i}{j}\sum_{p=0}^{n-i-j}\stirling{n-i-j}{p}{p}!
\frac{x^{i+p+1}}{(i+p+1)!} \nonumber \\
   & = &  
\binom{n}{i}
\sum_{j=k-i}^{n-i}
\binom{n-i}{j}\sum_{p=0}^{n-i-j}\stirling{n-i-j}{p}{p}!
\frac{x^{i+p+1}}{(i+p+1)!}, \label{eqn:gix}
\end{eqnarray}
and so its $(i+1)$-th derivative w.r.t. $x$ is
\begin{equation}
\label{eqn:i+1-th-diff}
g_{n;i}^{(i+1)}(x) = 
\binom{n}{i}
\sum_{j=k-i}^{n-i}
\binom{n-i}{j}\sum_{p=0}^{n-i-j}\stirling{n-i-j}{p}x^p 
:= \binom{n}{i}\gamma_{n;i}(x).
\end{equation}
From this we deduce that for 
\begin{equation}
\label{eqn:gixy}
g_i(x,y) := \sum_{n\geq 1}g_{n;i}(x)\frac{y^n}{n!}
\end{equation}
where $g_{n;i}(x)$ is given in (\ref{eqn:gix}), we have
\begin{eqnarray*}
g_i^{(i+1)}(x,y) & := & \frac{\partial^{i+1}}{{\partial x}^{i+1}}g_i(x,y) \\
  & = & \sum_{n\geq 1}g_{n;i}^{(i+1)}(x)\frac{y^n}{n!} \\
  & = & \sum_{n\geq 1}\binom{n}{i}\gamma_{n;i}(x)\frac{y^n}{n!} \\
  & = & \sum_{n\geq i}\binom{n}{i}\gamma_{n;i}(x)\frac{y^n}{n!} \\
  & = & \frac{y^i}{i!}\sum_{n\geq i}\gamma_{n;i}(x)\frac{y^{n-i}}{(n-i)!}.
\end{eqnarray*}

Before we continue, we need the following.
\begin{lemma}
\label{lmm:root-ord-parts}
For $m, N\in\nats$ we have 
\[
\sum_{i=1}^{N-m+1}\binom{N}{i}\stirling{N-i}{m-1} = m\stirling{N}{m}.
\]
\end{lemma}
\begin{proof}
Call an unordered partition of $[N]$ {\em rooted} if it has one
distinguished part, the {\em root} $r$. For each of the unordered
partitions of $[N]$ into $m$ parts, we have $m$ possible roots,
and so the number of rooted partitions of $[N]$ into $m$ parts is, 
on one hand, given by $m\stirling{N}{m}$.

On the other hand, we can start by choosing a root $r$ of valid cardinality
$i\in\{1,\ldots,N-m+1\}$, and then consider the $\stirling{N-i}{m-1}$ 
unordered partitions of the remaining subset of $[N]\subseteq r$. For each
$i$ this can be done in $\binom{N}{i}\stirling{N-i}{m-1}$ possible ways.
Adding these ways for all possible $i$ will give us the expression on the 
left in the stated equation.
\end{proof}
The coefficient $[x^h](\gamma_{n;i}(x))$ of $x^h$ in the polynomial 
$\gamma_{n;i}(x)$ 
defined in (\ref{eqn:i+1-th-diff}) is by direct tallying and the above
Lemma~\ref{lmm:root-ord-parts} given by
\begin{eqnarray*}
[x^h](\gamma_{n;i}(x)) & = & 
\sum_{q=k-i}^{n-i-h}\binom{n-i}{q}\stirling{n-i-q}{h} \\ 
 & = & \sum_{q=1}^{n-i-h}\binom{n-i}{q}\stirling{n-i-q}{h} -
  \sum_{q=1}^{k-i-1}\binom{n-i}{q}\stirling{n-i-q}{h} \\
 & = & \stirling{n-i}{h+1}(h+1) -
\sum_{q=1}^{k-i-1}\binom{n-i}{q}\stirling{n-i-q}{h}.
\end{eqnarray*}
In light of this, we can further write $\gamma_{n;i}(x)$ as
$\gamma_{n;i}(x) = \alpha_{n;i}(x) - \sum_{q=1}^{k-i-1}\alpha_{n;i,q}(x)$,
where 
\[
\alpha_{n;i}(x) := \sum_{h\geq 0}\stirling{n-i}{h+1}(h+1)x^h, \ \
\alpha_{n;i,q}(x) := \sum_{h\geq 0}\binom{n-i}{q}\stirling{n-i-q}{h}x^h.
\]
Defining the corresponding exponential series
\[
\alpha_i(x,y) := \sum_{n\geq i}\alpha_{n;i}(x)\frac{y^{n-i}}{(n-i)!}, \ \ 
\alpha_{i,q}(x,y) := \sum_{n\geq i}\alpha_{n;i,q}(x)\frac{y^{n-i}}{(n-i)!},
\]
we then get 
\begin{eqnarray*}
g_i^{(i+1)}(x,y) & = & 
\frac{y^i}{i!}\sum_{n\geq i}\gamma_{n;i}(x)\frac{y^{n-i}}{(n-i)!} \\
   & = & \frac{y^i}{i!}\sum_{n\geq i}\left(
\alpha_{n;i}(x) 
 - \sum_{q=1}^{k-i-1}\alpha_{n;i,q}(x)\right)\frac{y^{n-i}}{(n-i)!} \\
   & = & 
\frac{y^i}{i!}\left(\alpha_i(x,y) - \sum_{q=1}^{k-i-1}\alpha_{i;q}(x,y)\right).
\end{eqnarray*}
Up to a constant we obtain 
$\int \alpha_{n;i}(x)\,dx = \sum_{h\geq 0}\stirling{n-i}{h+1}x^{h+1}$,
and so by (\ref{eqn:Touch-St2}) we get
\[
\int\alpha_i(x,y)\,dx 
= \sum_{h,n\geq i}\stirling{n-i}{h+1}x^{h+1}\frac{y^{n-i}}{(n-i)!}
= \sum_{h,m\geq 0}\stirling{m}{h+1}x^{h+1}\frac{y^{m}}{m!}
= e^{x(e^y-1)}-1,
\]
and so by differentiating 
\begin{equation}
\label{eqn:alphai}
\alpha_i(x,y) = e^{x(e^y-1)}(e^y-1).
\end{equation}
Similarly, but with neither integration nor differentiation, we obtain 
by direct manipulation and by again (\ref{eqn:Touch-St2})
\begin{eqnarray*}
\alpha_{i;q}(x,y) & = & \sum_{n\geq i}\alpha_{n;i,q}(x)\frac{y^{n-i}}{(n-i)!} \\
   & = & \sum_{n\geq i}\left(\sum_{h\geq 0}\binom{n-i}{q}\stirling{n-i-q}{h}x^h
\right)\frac{y^{n-i}}{(n-i)!} \\
   & = & \sum_{m\geq 0}\left(\sum_{h\geq 0}\binom{m}{q}\stirling{m-q}{h}x^h
\right)\frac{y^m}{m!} \\
   & = & \sum_{m,h\geq 0}\binom{m}{q}\stirling{m-q}{h}\frac{y^m}{m!}x^h \\
   & = & 
\frac{y^q}{q!}\sum_{m,h\geq 0}\stirling{m-q}{h}\frac{y^{m-q}}{(m-q)!}x^h \\
   & = & \frac{y^q}{q!}e^{x(e^y-1)}.
\end{eqnarray*}
Consequentially, by the above and (\ref{eqn:alphai}) we then get
\begin{eqnarray*}
g_i^{(i+1)}(x,y)  
   & = & 
\frac{y^i}{i!}\left(\alpha_i(x,y) - \sum_{q=1}^{k-i-1}\alpha_{i;q}(x,y)\right) \\
   & = & \frac{y^i}{i!}\left(e^y - E_{k-i-1}(y)\right)e^{x(e^y-1)},
\end{eqnarray*}
where $E_m(x) = 1 + x + \cdots + \frac{x^m}{m!}$ is the $m$-th 
degree polynomial approximation of $e^y$.
By the the above and the defining sum of $g_i(x,y)$ in (\ref{eqn:gixy})
we have 
\begin{equation}
\label{eqn:gixy-cond}
g_i^{(i+1)}(x,y)  = \frac{y^i}{i!}\left(e^y - E_{k-i-1}(y)\right)e^{x(e^y-1)}, \ \ 
g_i(0,y) = g_i'(0,y) = \cdots = g_i^{(i)}(0,y) = 0 \mbox{ for all }y.
\end{equation}
The closed expression for $g_i(x,y)$ is uniquely determined by 
(\ref{eqn:gixy-cond}) and by integrating $i+1$ times we get
\begin{equation}
\label{eqn:gixy-form}
g_i(x,y) = 
\frac{y^i(e^y - E_{k-i-1}(y))}{i!(e^y-1)^{i+1}}\left(
e^{x(e^y-1)} - E_i(x(e^y-1))\right).
\end{equation}
By (\ref{eqn:gxy}), (\ref{eqn:g=sumgi}), (\ref{eqn:gixy}) 
and (\ref{eqn:gixy-form}) we then get
\[
g(x,y) 
= \sum_{n\geq 1}g_n(x)\frac{y^n}{n!} 
= \sum_{i=0}^{k-1}g_i(x,y) 
= \sum_{i=0}^{k-1}\frac{y^i(e^y - E_{k-i-1}(y))}{i!(e^y-1)^{i+1}}\left(
e^{x(e^y-1)} - E_i(x(e^y-1))\right), 
\]
and so by the above expression, (\ref{eqn:g-tilde-y}) and 
(\ref{eqn:g-2-g-tilde}) we have the following main theorem of this section.
\begin{theorem}
\label{thm:main-flag-func}
The exponential $\ell$-generating function for all the $\ell$-flags of all the
general permutahedra ${\cal{P}}_k^{\perp} = \{\Pi_{n-1}(k-1)\}_{n\geq k}$ from 
Definition~\ref{def:flag-poly-func} is given by 
\begin{eqnarray*}
{\tilde{\xi}}_k^{\ell}(\tilde{x},y) 
 & = & \sum_{n\geq 1,\tilde{s}}
{\tilde{\xi}}^{\ell}_{\Pi_{n-1}(k-1)}(\tilde{x})\frac{y^n}{n!} \\
 & = & \frac{1}{S}\sum_{i=0}^{k-1}
\frac{(x_1y)^i(e^{x_1y} - E_{k-i-1}({x_1y}))}{i!(e^{x_1y}-1)^{i+1}}
\left(e^{\frac{S}{x_1}\left(e^{x_1y}-1\right)} 
- E_i\left(\frac{S}{x_1}(e^{x_1y}-1)\right)
\right).
\end{eqnarray*}
where $S = x_2 + \cdots + x_{\ell} + 1$ 
and $E_m(x) = 1 + x + \cdots + \frac{x^m}{m!}$ is the $m$-th
polynomial approximation of $e^x$.
\end{theorem}
{\sc Remarks:}
(i) The question remains, whether one could possible further simplify 
the expression given in Theorem~\ref{thm:main-flag-func}. Since, however,
we are dealing with enumeration of OPP as
described in Theorem~\ref{thm:face-gen-perm}, in which the cardinalities
of the parts depend both on $n$ and $k$, it seems unlikely to the author
that a substantial simplification exists.
(ii) Letting $k=2$ we obtain the 
exponential $\ell$-generating function for all the $\ell$-flags of all the
general permutahedra $\{\Pi_{n-1}\}_{n\geq 2}$ as the following
\[
{\tilde{\xi}}_2^{\ell}(\tilde{x},y) = 
\frac{1}{x_2 + \cdots + x_{\ell} + 1}
\left(e^{\frac{x_2 + \cdots + x_{\ell} + 1}{x_1}\left(e^{x_1y}-1\right)}
- 1 - y(x_2 + \cdots + x_{\ell} + 1)\right),
\]
which is consistent with Proposition~\ref{prp:exp-gen-func}, 
when we note that the family $\{\Pi_{n-1}\}_{n\geq 1}$ there 
contains $\Pi_0$ for $n=1$, whereas in Theorem~\ref{thm:main-flag-func}
the family $\{\Pi_{n-1}(k-1)\}_{n\geq k}$ becomes $\{\Pi_{n-1}\}_{n\geq 2}$
omitting the singleton $\Pi_0$.

When $\ell = 1$ in Theorem~\ref{thm:main-flag-func} the sum
$x_2 + \cdots + x_{\ell}$ is empty  
and so we obtain the face function of all the general permutahedra
in the following.
\begin{corollary}
\label{cor:l=1}
The exponential generating function for all the faces of all the
general permutahedra $\{\Pi_{n-1}(k-1)\}_{n\geq k}$ from 
Definition~\ref{def:flag-poly-func} is given by 
\[
{\xi}_k(x,y) = {\tilde{\xi}}_k^{1}(x,y) = 
\sum_{i=0}^{k-1}\frac{(xy)^i(e^{xy} - E_{k-i-1}({xy}))}{i!(e^{xy}-1)^{i+1}}
\left(
e^{\left(\frac{e^{xy}-1}{x}\right)} - E_i\left(\frac{e^{xy}-1}{x}\right)
\right).
\]
where $E_m(x) = 1 + x + \cdots + \frac{x^m}{m!}$ is the $m$-th
polynomial approximation of $e^x$.
\end{corollary}

\subsection*{Acknowledgments}

The author would like to thank James F.~Lawrence
for many helpful discussions regarding the theory of polytopes
in general.


\begin{thebibliography}{10}

\bibitem{MinkPoly}
\newblock Geir Agnarsson,
\newblock The flag polynomial of the Minkowski sum of simplices.
\newblock {\em Annals of Combinatorics}, no.~3, {\bf 17}, 401 -- 426, (2013). 

\bibitem{Minkowski}
\newblock Geir Agnarsson and Walter Morris,
\newblock On Minkowski sum of simplices.
\newblock {\em Annals of Combinatorics}, {\bf 13}, 271 -- 287, (2009). 

\bibitem{Ard-Bene-Doker}
\newblock Federico Ardila; Carolina Benedetti; Jeffrey Doker, 
\newblock Matroid polytopes and their volumes. 
\newblock {\em Discrete Comput.~Geom.}, {\bf 43}, no.~4, 841 -- 854. (2010). 

\bibitem{Bell}
\newblock E.~T.~Bell, 
\newblock Exponential polynomials. 
\newblock {\em Ann.~of Math.~(2)},  {\bf 35}, no.~2, 258 -- 277, (1934). 

\bibitem{Sangwook-conf}
\newblock Sangwook Kim
\newblock Flag enumerations of matroid base polytopes.
\newblock 20th Annual International Conference on Formal Power Series 
and Algebraic Combinatorics (FPSAC 2008), 283 -- 294, 
\newblock {\em Discrete Math.~Theor.~Comput.~Sci.~Proc.}, 
AJ, Assoc.~Discrete Math.~Theor.~Comput.~Sci., Nancy, (2008). 

\bibitem{Sangwook-JCT}
\newblock Sangwook Kim,
\newblock Flag enumerations of matroid base polytopes. 
\newblock {\em J.~Combin.~Theory Ser.~A}, {\bf 117} no.~7, 928 -- 942, (2010). 

\bibitem{Postnikov}
\newblock Alexander Postnikov, 
\newblock Permutohedra, associahedra, and beyond. 
\newblock {\em Int.~Math.~Res.~Not.~IMRN}, no.~6, 1026 -- 1106, (2009). 

\bibitem{Post-Vic-Laur}
\newblock Alex Postnikov, Victor Reiner, Lauren Williams, 
\newblock Faces of generalized permutohedra.
\newblock {\em Doc.~Math.}, {\bf 13}, 207 -- 273, (2008). 

\bibitem{Rado}
\newblock Richard Rado
\newblock An inequality. 
\newblock {\em J.~London Math.~Soc.}, {\bf 27}, 1 -- 6, (1952). 

\bibitem{Ziegler}
\newblock G\"{u}nter M.~Ziegler,
\newblock Lectures on polytopes,
\newblock {\em Graduate Texts in Mathematics}, GTM -- 152,
Springer-Verlag Inc., New York, (1995).

\end{thebibliography}
\end{document}